\documentclass[11pt,letterpaper,reqno]{amsart}

\usepackage{amssymb}
\usepackage{amsmath}
\usepackage{amsthm}
\usepackage{bbm}
\usepackage{mathrsfs} 
\usepackage{esint} 
\usepackage{graphicx}
\usepackage{doi}
\usepackage{multirow}

\usepackage{bookmark}
\usepackage{hyperref}
\hypersetup{pdfstartview={FitH}}

\newtheorem{theorem}{Theorem}
\newtheorem{corollary}[theorem]{Corollary}
\newtheorem{remark}[theorem]{Remark}

\newtheorem{lemma}{Lemma}[section]

\newtheorem{problem}{Problem}

\numberwithin{equation}{section}

\renewcommand{\leq}{\leqslant}
\renewcommand{\geq}{\geqslant}

\newcommand{\R}{\mathbb{R}}
\newcommand{\Z}{\mathbb{Z}}
\newcommand{\ii}{\mathbbm{i}}

\begin{document}

\title[Coloring and density theorems for configurations of a given volume]{Coloring and density theorems for configurations of a given volume}

\author[Vjekoslav Kova\v{c}]{Vjekoslav Kova\v{c}}

\address{Department of Mathematics, Faculty of Science, University of Zagreb, Bijeni\v{c}ka cesta 30, 10000 Zagreb, Croatia}

\email{vjekovac@math.hr}




\begin{abstract}
This is a treatise on finite point configurations spanning a fixed volume to be found in a single color-class of an arbitrary finite (measurable) coloring of the Euclidean space $\mathbb{R}^n$, or in a single large measurable subset $A\subseteq\mathbb{R}^n$. More specifically, we study vertex-sets of simplices, rectangular boxes, and parallelotopes, attempting to make progress on several open problems posed in the 1970s and the 1980s. As one of the highlights, we give a negative answer to a question of Erd\H{o}s and Graham, by coloring the Euclidean plane $\mathbb{R}^2$ in $25$ colors without creating monochromatic rectangles of unit area. More generally, we construct a finite coloring of the Euclidean space $\mathbb{R}^n$ such that no color-class contains the $2^m$ vertices of any (possibly rotated) $m$-dimensional rectangular box of volume $1$. A positive result is still possible if rectangular boxes of merely sufficiently large volumes are sought in a single color-class of a finite measurable coloring of $\mathbb{R}^n$, and we establish it under an additional assumption $n\geq m+1$. Also, motivated by a question of Graham on reasonable bounds in his result on monochromatic axes-aligned right-angled $m$-dimensional simplices, we establish its measurable coloring and density variants with polylogarithmic bounds, again in dimensions $n\geq m+1$. Next, we generalize a result of Erd\H{o}s and Mauldin, by constructing an infinite measure set $A\subseteq\mathbb{R}^n$ such that every $n$-parallelotope with vertices in $A$ has volume strictly smaller than $1$. Finally, some results complementing the literature on isometric embeddings of hypercube graphs and on the hyperbolic analogue of the Hadwiger--Nelson problem also follow as byproducts of our approaches.
\end{abstract}

\maketitle

\tableofcontents


\section{Introduction}
\emph{Euclidean Ramsey theory} mainly investigates the existence of a given point configuration in an arbitrary finite coloring of the Euclidean space $\R^n$ or its subsets.
Its systematic study was initiated by Erd\H{o}s, Graham, Montgomery, Rothschild, Spencer, and Straus \cite{Eetal1,Eetal2,Eetal3} in the 1970s.
On the other hand, \emph{geometric measure theory} took a similar route in the 1980s and it started investigating point configurations that are found in arbitrary large measurable subsets of $\R^n$.
Here we call a set $A\subseteq\R^n$ \emph{large} if it has a positive density (as defined below) or, at least, a positive measure. 
Besides obvious connections to the Ramsey theory, the latter topic also shares the aesthetics with Szemer\'{e}di-type theorems from additive combinatorics. In fact, it can even be thought of as arithmetic combinatorics in a geometric setting, typically formulated in a continuous parameter rather than a discrete one.

The present paper studies coloring and density theorems for rather flexible point configurations, from which we only require that they span a body of a given shape and a given volume.
This particular subtopic has been initiated by Graham \cite{Gra80} and Erd\H{o}s \cite{Erd83:open}.
This treatise does not exhaust the subject, but rather concentrates on the simplest configurations, such as simplices, rectangular boxes, and parallelotopes, as they exhibit different behaviors already.
Our main goal is to show that many natural questions (or their variants) posed in the classical collections of open problems \cite{EG80,Erd83:open} (also see \cite{CFG91,GB15}) are manageable by the techniques of harmonic analysis that have been developed in the meantime.
However, quite a few basic questions are still left open.

Let us agree on the terminology. A \emph{finite coloring} of $\R^n$ (or its subset $S$) is any partition of $\R^n$ (or $S$) into finitely many sets $\mathscr{C}_1,\ldots,\mathscr{C}_r$, which are then called \emph{color-classes}.
A set $T$ is \emph{monochromatic} if $T\subseteq\mathscr{C}_j$ for some index $1\leq j\leq r$.
We say that the coloring is \emph{measurable} if each of the color-classes $\mathscr{C}_j$ is a Lebesgue-measurable set. 
Similarly, we say that the coloring is \emph{Jordan-measurable} if the boundary of $\mathscr{C}_j$ has Lebesgue measure $0$ for every index $j$.
Results that hold for an arbitrary finite (measurable) coloring of $\R^n$ (or an $S\subseteq\R^n$), perhaps with a prescribed number of colors, are sometimes referred to as \emph{coloring theorems}.

Next, the Lebesgue measure of a measurable set $A\subseteq\R^n$ (i.e., the length, the area, or, generally, the volume) will simply be written as $|A|$.
If $A$ is a measurable subset of the cube $[0,R]^n$ for some $R>0$, then we will refer to its \emph{density} as the number $|A|/R^n$.
On the other hand, for a general measurable $A\subseteq\R^n$ the relevant notion of largeness is the \emph{upper density}, defined as
\[ \bar{\textup{d}}_{n}(A) := \limsup_{R\to\infty} \frac{|A\cap[-R/2,R/2]^n|}{R^n}. \]
Results that hold for an arbitrary measurable $A\subseteq\R^n$ (or $A\subseteq[0,R]^n$) of strictly positive (upper) density are sometimes called \emph{density theorems}.
A slightly more convenient notion is the \emph{upper Banach density}, 
\[ \bar{\delta}_{n}(A) := \limsup_{R\to\infty} \sup_{x\in\R^n} \frac{|A\cap(x+[0,R]^n)|}{R^n}. \]
It is clearly always true that $\bar{\delta}_{n}(A)\geq\bar{\textup{d}}_{n}(A)$, so we make a density result stronger by assuming that $A$ merely has positive upper Banach density.

The tools from harmonic analysis employed in this work can be traced back to a paper by Bourgain \cite{B86:roth} and they have been strengthened or modified by a number of authors over the years. Some of the work that has been influential to the author of the present paper had been done by Lyall and Magyar \cite{LM16:prod,LM19:hypergraphs,LM20} and by Cook, Magyar, and Pramanik \cite{CMP15:roth}. 
Instead of listing the extensive literature we refer the reader to a partially expository paper \cite{K20:anisotrop} or a brief survey of recent results \cite{Kovac:survey}.
Harmonic analysis tools are well-suited for establishing density theorems in a fixed ambient dimension $n$, which, in turn, imply their measurable coloring variants (cf.\@ Theorems \ref{thm:simpldensity}, \ref{thm:boxdensity}, and \ref{thm:spacetimedensity} below). Indeed, if $\mathscr{C}_1,\ldots,\mathscr{C}_r$ is any measurable coloring of $\R^n$, then
\[ \bar{\delta}_{n}(\mathscr{C}_1) + \cdots + \bar{\delta}_{n}(\mathscr{C}_r) \geq \bar{\delta}_{n}(\R^n) = 1, \]
so we must have 
\[ \bar{\delta}_{n}(\mathscr{C}_j)\geq \frac{1}{r}>0 \] 
for at least one index $1\leq j\leq r$.
(An analogous claim holds for measurable colorings of the cube $[0,R]^n$.)
Contrapositively, if we construct a measurable counterexample to a potential coloring result, then the corresponding density result is not possible either (cf.\@ Theorems \ref{thm:colorR2}, \ref{thm:colorRn}, \ref{thm:negparallel}, and \ref{thm:hypercubes} below).

Admittedly, the measurability assumption is a seriously restricting one.
There are many examples in the literature where the coloring result becomes significantly easier when only measurable colorings are considered; see Soifer's book \cite[Ch.~9]{Soi24:book}.
For instance, the well-known \emph{Hadwiger--Nelson problem} asks for a minimal number of colors needed to color $\R^2$ so that all pairs of points at distance $1$ apart receive different colors.
It is easy to see that $7$ colors suffice.
The exact minimal number has naturally been called the \emph{chromatic number of the plane}.
Back in 1981, Falconer \cite{Fal81} showed that the measurable chromatic number of $\R^2$ is at least $5$, but it took another $37$ years until de Grey \cite{Grey18} proved the same for the ordinary chromatic number.
Also, from an old result by Furstenberg, Katznelson, and Weiss \cite{FKW90:dist} it follows that a measurable set $A\subseteq\R^2$ satisfying $\bar{\textup{d}}_{2}(A)>0$ contains pairs of points at every sufficiently large distance from each other.
This has been reproved independently by Falconer and Marstrand \cite{FM86:dist} and also by Bourgain \cite{B86:roth}.  
As opposed to that, only recently Davies \cite{Dav24} solved a 1994 conjecture of Rosenfeld \cite[\S 60.1]{Soi24:book} (popularized by Erd\H{o}s), by showing that there exists a pair of points at an odd distance from each other for every (not necessarily measurable) finite coloring of $\R^2$. 
This was generalized to prime and polynomial distances by Davies, McCarty, and Pilipczuk \cite{DCP23}. A question by Kahle about distances of the form $n!$ \cite[Problem 60.14]{Soi24:book} and the one by Soifer on distances of the form $2^n$ \cite[Problem 60.16]{Soi24:book} remain open.


\subsection{Triangles and simplices}
\label{subsec:simplices}
Graham \cite{Gra80} answered positively a question of Gurevich, who asked whether, for positive integers $2\leq m\leq n$ and every finite coloring of $\R^n$, there exists an $m$-dimensional simplex of volume $1$ with all of its $m+1$ vertices colored the same \cite[p.~89]{Gra80}.
In fact, Graham could choose the simplices to be right-angled with axes-aligned perpendicular edges, i.e., the edges coming from a single vertex are parallel to the coordinate axes.
Graham combined his observations with a general product theorem for colorings \cite[Theorem 3]{Gra80}, which he developed particularly for this purpose; also see the generalization that he proved jointly with Spencer \cite{GS79}. His most general coloring theorem for simplices, namely \cite[Theorem 4]{Gra80}, reads as follows.

\begin{quote}
\emph{For all finite colorings of $\R^n$, some color-class contains, for all $V>0$ and lines $\ell_1,\ldots,\ell_n$ which span $\R^n$, a simplex having volume $V$ and edges through one vertex parallel to the $\ell_i$.}
\end{quote}

Indeed, the ``slanted'' variant of this result is clearly equivalent to the axes-aligned one, simply by applying a volume-preserving linear transformation. 
Several variants of Graham's result were studied in \cite{DJ10,AR12,AC14}.

By inspecting Graham's proof (already for triangles) one notices that he first works on a huge part of the integer lattice $\Z^n$ \cite[Theorem 1]{Gra80}, the size of which is an iterated van der Waerden number in the number of colors used. Then he scales the lattice to obtain the result in $\R^n$, but also wonders \cite[p.~97]{Gra80} if a proof using only a reasonable number of integer points (or, strictly speaking, only reasonably large integer-point triangles) is possible. 
A natural continuous-parameter version of Graham's question is the one that considers only colorings of a sufficiently large cube $[0,R]^n\subseteq\R^n$.

\begin{problem}
Let $m,n$ be fixed positive integers such that $n\geq m\geq 2$.
Is there a reasonable lower bound (i.e., not of the Ackermann type) on the number $R=R(r)\in(0,\infty)$ such that, for every coloring of the $n$-cube $[0,R]^n$ in $r$ colors, there exists a right-angled $m$-dimensional simplex of unit volume with monochromatic vertices?
\end{problem}

We are able to give the positive answer for measurable colorings only and under the additional assumption $n\geq m+1$ in part (b) of the following theorem. At the same time, we establish its corresponding density variant in part (a).

\begin{theorem}\label{thm:simpldensity}
For every integer $m\geq2$ there exists a constant $C_m\in(0,\infty)$ with the following properties for every $n\geq m+1$.
\begin{itemize}
\item[(a)] If $R\in(1,\infty)$ and $A\subseteq[0,R]^n$ is a measurable set with density
\begin{equation}\label{eq:conddens}
\frac{|A|}{R^n} \geq \Big(\frac{C_m}{\log R}\Big)^{1/(9m^2)},
\end{equation}
then $A$ contains $m+1$ vertices of a right-angled $m$-dimensional simplex of unit volume.
\item[(b)] If $R\in(0,\infty)$ and if the $n$-cube $[0,R]^n$ is measurably colored in $r$ colors, then
\[ R \geq \exp\big(C_m r^{9m^2}\big) \]
is sufficient to guarantee the existence of a right-angled $m$-dimensional simplex of unit volume with monochromatic vertices.
\end{itemize}
In both parts of the theorem, the $m$-simplex can be chosen such that $m-1$ of its edges from the right-angled vertex are parallel to the coordinate vectors $\mathbbm{e}_1,\ldots,\mathbbm{e}_{m-1}$, while the remaining edge is parallel to the linear span of $\mathbbm{e}_m,\ldots,\mathbbm{e}_n$. 
\end{theorem}

\begin{figure}
\includegraphics[width=0.5\linewidth]{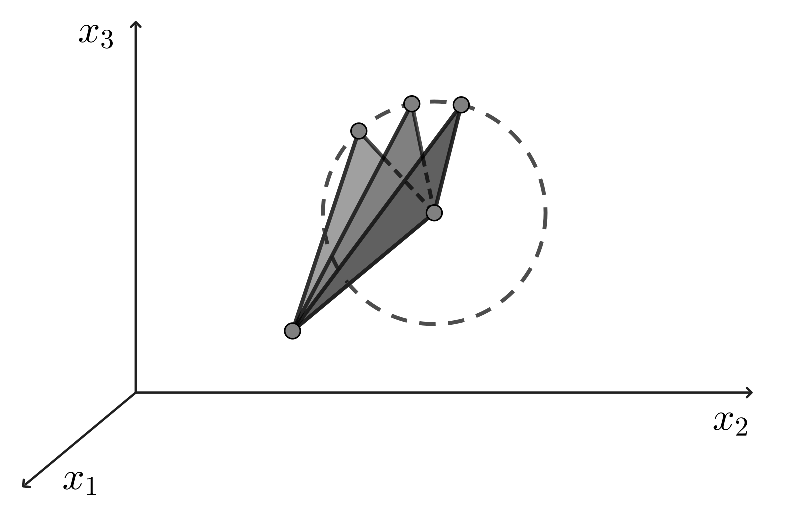}
\caption{Illustration of Theorem \ref{thm:simpldensity}.}
\label{fig:colfig7}
\end{figure}

Some possible positions of such right triangles in $\R^3$ are depicted in Figure \ref{fig:colfig7}. Clearly, part (b) is a consequence of part (a), since, by our introductory comments, at least one color-class has density $\delta$ at least $1/r$ and since condition \eqref{eq:conddens} can be rewritten as
\[ R \geq \exp\big(C_m \delta^{-9m^2}\big) \]
in terms of the density $\delta$.
The emphasis in Theorem \ref{thm:simpldensity} is on its effective bounds: a polylogarithmic one in part (a) and an exponential one (in the power of $r$) in part (b). 
The order of quantifiers is important: we want a density theorem in a fixed dimension $n$. It is not difficult to fix a number $\delta>0$ and find a sufficiently large $n$ such that the result holds for regular simplices in subsets of $\R^n$ with density at least $\delta$. In this setting, quantitative interdependence of $n$ and $\delta$ has been studied by Castro-Silva, de Oliveira Filho, Slot, and Vallentin \cite{COSV22}.

Graham suggested (without sharing many details) the possibility of a density variant of his discrete result \cite[Theorem 1]{Gra80} in Concluding Remarks \cite[p.~96]{Gra80}, but he mentioned general triangles only, i.e., not axes-aligned or even right-angled ones, which makes the problem easier.
Indeed, the corresponding analogue in the plane $\R^2$ is, however, quite easy: it is an old observation of Erd\H{o}s, first published as an exercise in the popular magazine \emph{Matematikai Lapok}; see the comments in \cite[p.~122]{Erd78}.

\begin{quote}
\emph{A measurable set $A\subseteq\R^2$ of infinite area (and even an unbounded set of positive area) necessarily contains the vertices of a triangle of area $1$.}
\end{quote}

The proof is a simple consequence of the one-dimensional Steinhaus theorem \cite{Ste20} and it was elaborated on in \cite[p.~182]{CFG91}.
The corresponding problem for right-angled triangles was mentioned by Erd\H{o}s \cite[p.~324]{Erd83:open} and attributed to him and Mauldin.

\begin{problem}[Erd\H{o}s and Mauldin \cite{Erd83:open}]
\label{prob:triangleinfinity}
Let $A\subseteq\mathbb{R}^2$ be a measurable set with infinite area. Must $A$ contain the vertices of a right triangle of area $1$?
\checkmark
\end{problem}

Very recently Koizumi \cite{Koi25} answered this question positively, claiming to have been motivated by the ideas from a recent paper by Predojevi\'{c} and the present author \cite{KP25}. The techniques from \cite{KP25,Koi25} rely on the Lebesgue density theorem and thus cannot recover any of the inherently quantitative results of the present paper. Another four questions by Erd\H{o}s and Mauldin (on other shapes) similar to Problem \ref{prob:triangleinfinity} can be found in \cite[p.~323--324]{Erd83:open} or gathered under a single problem on Bloom's website \emph{Erd\H{o}s problems} \cite[\#353]{EP}.
They were all answered either in \cite{KP25} or in \cite{Koi25}. 

In the last problem we cannot also ask from the triangle to be axes-aligned.
Indeed, the example of a rotated thin infinite strip/tube shows that, if $A\subseteq\R^n$ merely has infinite volume, then the perpendicular sides of right-angled unit triangles/simplices with vertices in $A$ can avoid any prescribed tuple of orthogonal directions.
The same example also shows that the simplices with vertices in a set $A$ of merely positive area can all have disproportional ratios of their edge-lengths.
For this reason we find non-trivial the following consequence of Theorem \ref{thm:simpldensity}(a), which is quantitative and it deals with sets of positive upper Banach density. 

\begin{corollary}\label{cor:simpldensity}
Let $m\geq2$ and $n\geq m+1$ be positive integers. There exists a constant $C'_m\in(0,\infty)$, depending only on $m$, with the following property: if a measurable set $A\subseteq\R^n$ satisfies $\bar{\delta}_{n}(A)>0$, then for every number $V>0$ there exists a right-angled $m$-dimensional simplex of $m$-volume $V$ with all $m+1$ vertices in $A$ and with the ratio of lengths of any two of its perpendicular edges being at most $\exp(C'_m\bar{\delta}_{n}(A)^{-9m^2})$. 
\end{corollary}

An interesting related problem is an old question of Erd\H{o}s on triangles in planar sets $A$ of merely sufficiently large finite area. 

\begin{problem}[Erd\H{o}s \cite{Erd78}]
\label{prob:trianglelarge}
Is it true that there is an absolute constant $C$ so that, if $A\subseteq\R^2$ has area greater than $C$, then $A$ contains the vertices of a triangle of area $1$? 
\end{problem}

It seems that Erd\H{o}s asked this question first in \cite[p.~122--123]{Erd78}, and then repeated it in \cite[p.~30]{Erd:Scottish} and \cite[p.~323]{Erd83:open}.
Problem \ref{prob:trianglelarge} has since been featured in numerous collections of open problems, such as the book \emph{Unsolved problems in geometry} by Croft, Falconer, and Guy \cite[Problem G13]{CFG91}, the website \emph{Erd\H{o}s problems} \cite[\#352]{EP}, or the expository paper by Mauldin \cite{Mau02}.
The techniques in this paper are insufficient to detect configurations in ``scattered'' sets that are only known to have large measure.


\subsection{Rectangles and rectangular boxes}
\label{subsec:rectangles}
Graham \cite{Gra80} naturally wondered if his result that we discussed in the previous section
\begin{quote}
(\ldots) \emph{can be generalized to configurations which are not simplexes.} \cite[p.~96]{Gra80}.
\end{quote}
Erd\H{o}s and Graham posed a natural follow-up question in their well-known problem book \cite{EG80}.
\begin{quote}
\emph{The question is: Is this also true for rectangles? Or perhaps parallelograms?} \cite[p.~15]{EG80}.
\end{quote}
To the best of the author's knowledge, these problems, either for rectangles or for parallelograms, have not been addressed in the literature so far.
They were explicitly mentioned as being open in the papers by Adhikari and Rath \cite{AR03,AR12} and in \emph{Mathematical Reviews} MR0558877 by K. Steffens and MR2106573 by S. Oates-Williams.
They were also recently stated among the open questions in the 2015 edition of the book \emph{Rudiments of Ramsey theory} by Graham and Butler \cite[p.~56]{GB15}.

Let us first discuss the variant for rectangles. 
The following standalone wording of the problem is taken from Bloom's website \emph{Erd\H{o}s problems} \cite[\#189]{EP}.

\begin{problem}[Erd\H{o}s and Graham \cite{Gra80,EG80}]
\label{prob:main}
If $\mathbb{R}^2$ is finitely colored then must there exist some color-class which contains the vertices of a rectangle of every given area? \checkmark
\end{problem}

Note that the rectangles can be arbitrarily rotated.
We are able to show that Problem \ref{prob:main} has a negative answer.

\begin{theorem}\label{thm:colorR2}
There exists a Jordan-measurable coloring of the plane in $25$ colors such that no color-class contains the vertices of a rectangle of area $1$. 
\end{theorem}

With a bit more work we will also establish a natural generalization to higher dimensions; see Theorem \ref{thm:colorRn} below. However, we have singled out the two-dimensional case in Theorem \ref{thm:colorR2} above, because its proof is elegant and constructive.

\begin{theorem}\label{thm:colorRn}
For every positive integer $n$ there exists a finite Jordan-measurable coloring of the Euclidean space $\mathbb{R}^n$ with the following property: for every positive integer $m\leq n$ there is no $m$-dimensional rectangular box of $m$-volume equal to $1$ with all of its $2^m$ vertices colored the same. 
\end{theorem}

The order of quantifiers is important again. The analogue of Theorem \ref{thm:colorRn} would not hold and Theorem \ref{thm:boxdensity} below would become trivial if the ambient dimension $n$ could depend on the number of colors used. One could find sufficiently large $n$ simply by the known fact that vertices of a unit cube (or any other rectangular box) form a so-called \emph{Ramsey configuration}; see \cite{Eetal1}.

The minimal number of colors enabling a coloring of $\mathbb{R}^n$ with the property from Theorem \ref{thm:colorRn} grows at least exponentially in $n$.
This is seen by taking $m=1$ and consulting the literature on the higher-dimensional Hadwiger--Nelson problem, a.k.a.\@ the chromatic number of $\mathbb{R}^n$; see \cite{FW81}.
On the other hand, we will see that the number of colors needed in the proof presented in Section \ref{sec:Rn} grows superexponentially in $n$. It would be interesting to obtain reasonably sharp bounds on the cardinality of the minimal coloring.

After the negative result in Theorem \ref{thm:colorRn} it is perhaps still possible that, for every finite coloring of $\R^n$, there exist monochromatic $m$-dimensional rectangular boxes of all sufficiently large $m$-volumes.
We are able to show that this is indeed the case as long as $n\geq m+1$ and only measurable colorings are considered.
In fact, we can even prove a density version of that result thanks to the fact that the techniques for studying large copies of finite configurations in sets of positive density have been developed substantially over the last $40$ years (cf.\@ the literature mentioned at the beginning of the introductory section).

\begin{theorem}\label{thm:boxdensity}
Take positive integers $m$ and $n$ such that $n\geq m+1$.
\begin{itemize}
\item[(a)] If $A\subseteq\R^n$ is a measurable set such that $\bar{\delta}_n(A)>0$, then there exists a number $V_0>0$ (depending on the set $A$) such that for every number $V\geq V_0$ there exists an $m$-dimensional rectangular box of $m$-volume $V$ with all $2^m$ vertices in the set $A$.
\item[(b)] For every finite measurable coloring of $\R^n$ there exist a color-class $\mathscr{C}$ and a number $V_0>0$ such that for every number $V\geq V_0$ there exists an $m$-dimensional rectangular box of $m$-volume $V$ with all vertices in $\mathscr{C}$.
\end{itemize}
In both parts of the theorem, rectangular boxes can be chosen to have $m-1$ edges parallel to the first $m-1$ coordinate vectors and the last edge parallel to the linear span of the remaining $n-m+1$ coordinate vectors.
\end{theorem}

Theorem \ref{thm:boxdensity} is only new when the ambient dimension $n$ satisfies $m+1\leq n\leq 2m-1$.
It follows from the existing literature in the cases $n\geq 2m$, as Lyall and Magyar \cite[Theorem 1.1(i)]{LM19:hypergraphs} showed that a set $A\subseteq\R^{2m}$ of positive upper Banach density contains the vertices of an $m$-cube with every sufficiently large edge-length.
The critical case $n=m\geq2$ is not covered by our theorem and it remains open.
The analogous result for $n=m=1$ clearly fails, as is seen by taking $A=(-1/5,1/5)+\mathbb{Z}$.
Thus, the following problem remains.

\begin{problem}\label{prob:5}
If $n\geq2$ and $\mathbb{R}^n$ is measurably colored in finitely many colors, must there exist a color-class which contains the $2^n$ vertices of a rectangular box of every sufficiently large volume?
\end{problem}

It is hard to say if Problem \ref{prob:5} is within the reach of harmonic analysis techniques, but it is expected to be difficult. Namely, when one sets up the harmonic analysis ``machinery'' it seems that the relevant analytical objects are just slightly ``too singular'' to be handled using the currently known ideas.
Also, its case $n=2$ is formally similar to Problem \ref{prob:7} below, for triangles in the plane, which is open as well. 
It is also unclear if the measurability condition can be removed from the statement of Theorem \ref{thm:boxdensity}(b).


\subsection{Parallelograms and parallelotopes}
The situation changes radically if the boxes are ``slanted'', i.e., replaced with parallelograms, parallelepipeds, or, more generally, parallelotopes.
Currently, we are not able to solve the parallelogram variant of Problem \ref{prob:main}.

\begin{problem}[Erd\H{o}s and Graham \cite{Gra80,EG80}]
\label{prob:parallel}
If $\mathbb{R}^2$ is finitely colored then must there exist some color-class which contains the vertices of a parallelogram of every given area?
\end{problem}

Unable to answer the question in Problem \ref{prob:parallel}, we only state a partial result that will be obtained by easily modifying the ideas used for Theorem \ref{thm:colorR2}.

\begin{theorem}\label{thm:negparallel}
Suppose that we are given finitely many lines, $\ell_1,\ldots,\ell_m\subset\R^2$, and a positive number $\varepsilon$. There exists a Jordan-measurable coloring of $\mathbb{R}^2$ such that there is no parallelogram of area $1$ with monochromatic vertices that, additionally, has one side parallel to some line $\ell_i$ or it has all angles greater than $\varepsilon$.
\end{theorem}

In other words, one should concentrate on finding almost degenerate parallelograms and infinitely many directions should be considered.
For these reasons Problem \ref{prob:parallel} could be difficult, especially if it turns out that the claim has a positive answer.

As a related observation, Erd\H{o}s remarked \cite[p.~324]{Erd83:open} (without details or a reference) that he and Mauldin constructed a set $A\subseteq\R^2$ of infinite area that does not contain vertices of a parallelogram of area $1$ (also see the comments under Problem \cite[\#353]{EP}).
Even if this has no implications to finite colorings of $\R^2$ or to positive density subsets of $\R^2$, it still hints that sets that avoid the vertices of parallelograms of unit area can be quite large.
Without knowing which set Erd\H{o}s and Mauldin had in mind, we observe that the example can be quite simple: an $(x,y)$-region between the positive parts of the coordinate axes $x=0$ and $y=0$ and the hyperbola $2xy=1$.
In fact, we will easily prove the following stronger higher-dimensional result.

\begin{theorem}\label{thm:parallel}
Let $n\geq2$ be a positive integer and take $\varepsilon>0$. There exists a Jordan-measurable set (i.e., a set with boundary of measure $0$) $A\subseteq\mathbb{R}^n$ of infinite volume such that every $n$-dimensional parallelotope with all $2^n$ vertices in $A$ has volume less than $\varepsilon$.
\end{theorem}

One has to note that Theorem \ref{thm:parallel} is a rather simple result, using a more straightforward construction than Theorem \ref{thm:colorRn}.


\subsection{Distance graphs in the plane}
\label{subsec:dgraphs}
Here we discuss two lines of research that, at first, do not seem to be related to the paper topic, as they will be concerned with certain non-rigid configurations, as opposed to polytopes.
However, it will turn out, quite surprisingly, that Theorem \ref{thm:hypercubes} below can be proved simply by modifying the proof of Theorem \ref{thm:colorR2} from Subsection \ref{subsec:rectangles}, while the configuration of interest in Theorem \ref{thm:bmk} below will be obtained by taking two out of three vertices of Graham's axes-aligned triangle from Subsection \ref{subsec:simplices} and performing a clever change of variables appearing in \cite{BMK17} and \cite{Dav24a}.
Thus, the study of these flexible configurations fits naturally to the present paper as well and we will be able to reuse the same techniques applied to our previous results, both positive and negative ones.

Lyall and Magyar \cite{LM20} initiated the study of embeddings of distance graphs in measurable sets $A\subseteq\R^n$ of positive upper Banach density.
The most definite embedding result is available for distance trees, which were also studied (in similar contexts) in \cite{Bul18,IT19,LMtr20}, so let us spend a few words to formulate it here.

Let $\mathcal{T}=(V,E)$ be a finite tree with the set of vertices $V$ and the set of edges $E$. 
To every edge $k\in E$ we assign a number $a_k\in(0,\infty)$ interpreted as \emph{edge-length}, so we call the resulting structure a \emph{distance tree}. 
It follows from a more general result by Lyall and Magyar \cite[Theorem 2(i)]{LM20} that
\begin{quote}
\emph{for every distance tree $\mathcal{T}$ and every set $A\subseteq\R^2$ of positive upper Banach density, there exists a number $\lambda_0>0$ (depending on $\mathcal{T}$ and $A$) such that for every number $\lambda\geq\lambda_0$ the set $A$ contains the vertices of an isometric copy of $\mathcal{T}$ scaled by factor $\lambda$.}
\end{quote}
Note that a non-scaled tree (or any other graph) need not be isometrically embeddable in $A$, simply because there is no reason why $A$ should contain two points at distance $1$ apart (cf.\@ the Hadwiger--Nelson problem mentioned before).
The above result is also dimensionally sharp: sufficiently large distances need not be found in a positive density subset of $\R$; take $A=[-1/5,1/5]+\Z$.


\subsubsection{Embeddings of hypercube graphs}
Embeddings of more general distance graphs are much more complicated and only partially resolved.
Whenever the graph has a cycle, one can choose its edge-lengths to violate the triangle inequality, so it will not be embeddable in the Euclidean space. 
Lyall and Magyar \cite{LM20} resolved this problem by starting from a distance graph that already is embedded in some other Euclidean space.
Then they gave sufficient conditions for it to have all large dilates found in every set $A\subseteq\R^d$ of positive upper Banach density. However, even when their \cite[Theorem 2]{LM20} applies, it is rarely known what the minimal ambient dimension $d$ can be.
Surprisingly, this is open already in the case of the equilateral (or any other non-degenerate) triangle.

\begin{problem}[implicit in \cite{FKW90:dist}, stated as {\cite[Problem G14]{CFG91}}]\label{prob:7}
Do all sets $A\subseteq\R^2$ of positive upper density contain a congruent copy of the vertices of every sufficiently large equilateral triangle?
\end{problem}

Note that the corresponding result for triangles in $\R^3$ is known, as it is a special case of Bourgain's \cite[Theorem 2]{B86:roth}.

Here we study a particular case of a graph, which is called a \emph{hypercube graph} in the combinatorial sense, but as a distance graph it is better thought of as a $1$-skeleton of an $n$-dimensional box
\begin{equation}\label{eq:nbox} 
[0,a_1]\times[0,a_2]\times\cdots\times[0,a_n].
\end{equation}
Its embedding in $\R^2$ is simply a set of $2^n$ mutually distinct points
\begin{equation}\label{eq:1skeleton}
z + r_1 u_1 + r_2 u_2 + \cdots + r_n u_n \quad\text{for } (r_1,r_2,\ldots,r_n)\in\{0,1\}^n
\end{equation}
such that $z\in\R^2$ and $u_1,\ldots,u_n\in\R^2$ are vectors with lengths $|u_k|=a_k$ for $1\leq k \leq n$; see Figure \ref{fig:colfig5}. 

\begin{figure}
\includegraphics[width=0.65\linewidth]{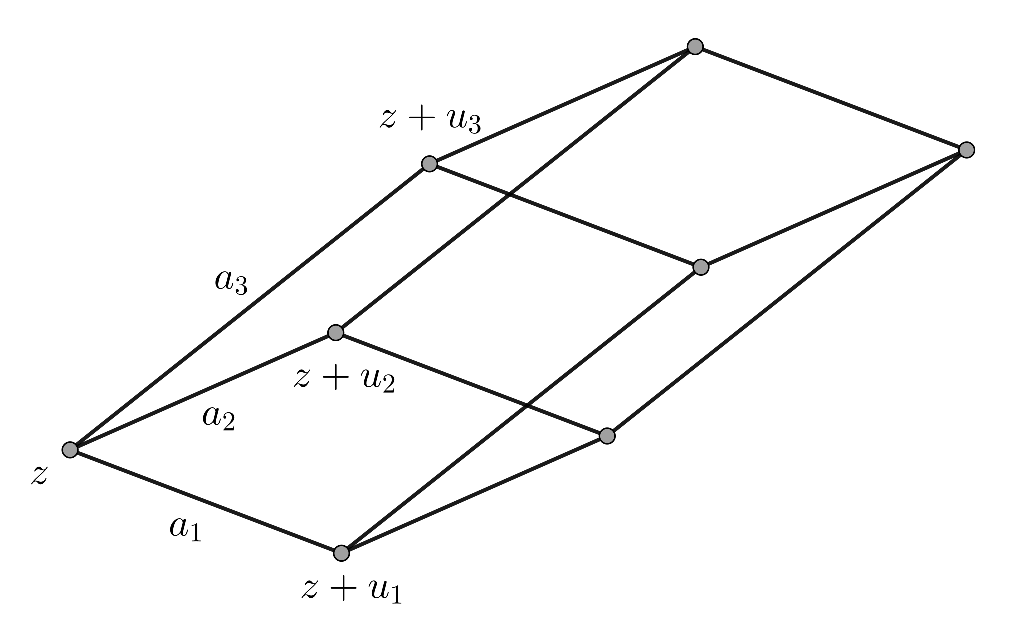}
\caption{Embedding of a $1$-skeleton of an $n$-box.}
\label{fig:colfig5}
\end{figure}

This case of an embedding in $\R^2$ was not covered by \cite[Theorem 2]{LM20}; their result would require the ambient dimension to be at least $n+1$.
Still, Predojevi\'{c} and the author \cite[Theorem 1 and Remark 3]{KP23} showed that 
\begin{quote}
\emph{a measurable subset $A\subseteq\R^2$ of positive upper Banach density contains embeddings of all sufficiently large dilates of a fixed $1$-skeleton of an $n$-box; for instance, in $A$ we can find points \eqref{eq:1skeleton} with $|u_k|=a_k=\lambda$ for all sufficiently large numbers $\lambda\in(0,\infty)$.}
\end{quote}
This result generalizes the aforementioned result on distance trees, since every tree is a subgraph of a hypercube graph.
A natural question is what happens when we assume a more flexible constraint relating the edge-lengths than $a_1=a_2=\cdots=a_n=\lambda$, such as the product of the numbers $a_k$ being $1$.

\begin{theorem}\label{thm:hypercubes}
Fix a positive integer $n$.
There exists a Jordan-measurable finite coloring of $\mathbb{R}^2$ such that no color-class contains an embedding \eqref{eq:1skeleton} of the $1$-skeleton of an $n$-box \eqref{eq:nbox} satisfying $a_1\cdots a_n=1$.
\end{theorem}

In particular, a measurable set $A\subseteq\R^2$ with $\bar{\delta}_2(A)>0$ need not contain such an embedding.
Theorem \ref{thm:hypercubes} will be shown by a simple modification of the construction used to prove Theorem \ref{thm:colorR2} above.
This might sound a bit surprising, since the quantity $a_1\cdots a_n$ is the volume of the original box \eqref{eq:nbox}, but it has no such interpretation for an embedding of its $1$-skeleton \eqref{eq:1skeleton}.
This, however, leads to structural differences between embeddings of parallelotopes of unit volume and their partial skeletons with edge-lengths multiplying to $1$, but we provided negative results for both. Unit volume rectangular boxes lie at their intersection, which enabled us to obtain the strongest result on their avoidance, namely Theorem \ref{thm:colorRn}.


\subsubsection{Hyperbolic embeddings}
Bardestani and Mallahi-Karai \cite{BMK17} studied a generalization of the Hadwiger--Nelson problem in which the points $v$ and $w$ of the same color are not allowed to satisfy the equation $Q(v-w)=1$, where $Q$ is a quadratic form over some local field.
Under the additional measurability assumption, they were able to characterize when the corresponding chromatic number is infinite.
Bardestani also informed the author that an easy modification of the argument from \cite{BMK17}, which in turn uses a bound on the so-called independence ratio by Bachoc, DeCorte, de Oliveira Filho, and Vallentin \cite{BCFV14}, also implies a density variant of the same result in the plane $\R^2$ in the case of real hyperbolic forms; see Theorem \ref{thm:bmk} below.
Davies \cite{Dav24a} was motivated by a few unresolved questions in the particular hyperbolic case
\[ Q\colon\R^2\times\R^2 \to \R,\quad Q((x,y),(x',y')) = (x-x')^2 - (y-y')^2 \]
and its higher-dimensional space-time generalizations.
He observed that Graham's result on right triangles is relevant in that matter and it already implies that the corresponding (non-measurable) chromatic number is infinite as well.
Namely, if $(u,v)$, $(u',v')$, $(u',v)$ are vertices of an area $1$ right triangle with $u>u'$ and $v<v'$, then the points
\[ \begin{pmatrix} x \\ y \end{pmatrix}
= \frac{1}{2\sqrt{2}} \begin{pmatrix} 1 & -1 \\ 1 & 1 \end{pmatrix} 
\begin{pmatrix} u \\ v \end{pmatrix},\quad
\begin{pmatrix} x' \\ y' \end{pmatrix}
= \frac{1}{2\sqrt{2}} \begin{pmatrix} 1 & -1 \\ 1 & 1 \end{pmatrix} 
\begin{pmatrix} u' \\ v' \end{pmatrix} \]
satisfy $Q((x,y),(x',y'))=1$.

It is natural to consider a coloring of a large square $[0,R]^2$ and impose bounds on $R$ sufficient to find equally colored points $(x,y)$ and $(x',y')$ such that
\begin{equation}\label{eq:hyperdiffer}
(x-x')^2 - (y-y')^2 = 1.
\end{equation}
The arguments of Bardestani and Mallahi-Karai \cite[p.~325--327]{BMK17} are also quantitative and they showed that this is true for measurable colorings in $r$ colors under the condition of the form $R\geq\exp(Cr)$. 
Let us formulate this result in the spirit of Theorem \ref{thm:simpldensity}.

\begin{theorem}[Bardestani and Mallahi-Karai \cite{BMK17}]\label{thm:bmk}
There exists a constant $C\in(0,\infty)$ with the following properties.
\begin{itemize}
\item[(a)] If $R\in(1,\infty)$ and $A\subseteq[0,R]^2$ is a measurable set with density
\[ \frac{|A|}{R^2} \geq \frac{C}{\log R}, \]
then there exist $(x,y),(x',y')\in A$ such that \eqref{eq:hyperdiffer} holds.
\item[(b)] If $R\in(0,\infty)$ and if the square $[0,R]^2$ is measurably colored in $r$ colors, then
\[ R \geq e^{C r} \]
guarantees that there exist equally colored points $(x,y)$ and $(x',y')$ satisfying \eqref{eq:hyperdiffer}.
\end{itemize}
\end{theorem}

In comparison, the argument of Davies \cite{Dav24a} based on Graham's theorem applies to non-measurable colorings too, but it requires an Ackermann-type bound on $R$ in terms of $r$. 

We remark, once again, that Theorem \ref{thm:bmk} is not an original result that we show here. As we have mentioned, it was not explicitly formulated in \cite{BMK17}, but it can be deduced rather easily from that paper, as observed by its authors.
Our techniques are better suited to larger patterns, so we show the following generalization to a configuration of $n+1$ points, with a slightly weaker bound for $n=1$ than Theorem \ref{thm:bmk}.
Let $\mathscr{H}$ denote the standard hyperbola in $\R^2$ with the Cartesian equation 
\[ \mathscr{H}: \quad x^2-y^2=1. \]
It will replace the standard unit circle, $x^2+y^2=1$, which plays a role in the usual Euclidean graph embedding theorems.

\begin{theorem}\label{thm:spacetimedensity}
For every positive integer $n$ and distinct numbers $a_1,\ldots,a_n\in(0,\infty)$ there exists a constant $C\in(0,\infty)$ with the following properties.
\begin{itemize}
\item[(a)] If $R\in(1,\infty)$ and $A\subseteq[0,R]^2$ is a measurable set with density
\[ \frac{|A|}{R^2} \geq \Big(\frac{C}{\log R}\Big)^{1/(2n+1)}, \]
then there exist $z\in\R^2$ and vectors $v_1,\ldots,v_n\in\mathscr{H}$ such that the $n+1$ points
\begin{equation}\label{eq:hyperpoints}
z,\, z+a_1v_1,\, \ldots,\, z+a_nv_n
\end{equation}
all belong to the set $A$.
\item[(b)] If $R\in(0,\infty)$ and if the square $[0,R]^2$ is measurably colored in $r$ colors, then
\[ R \geq \exp\big(C r^{2n+1}\big) \]
guarantees that there exist $n+1$ equally colored points of the form \eqref{eq:hyperpoints} for some $z\in\R^2$ and vectors $v_1,\ldots,v_n\in\mathscr{H}$.
\end{itemize}
\end{theorem}

\begin{figure}
\includegraphics[width=0.55\linewidth]{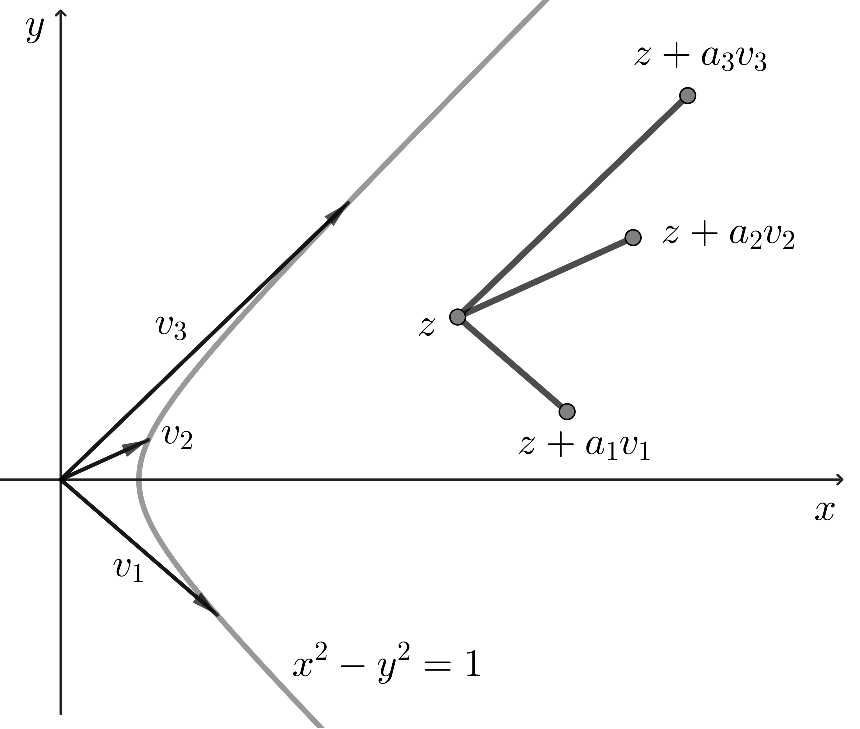}
\caption{Illustration of Theorem \ref{thm:spacetimedensity}.}
\label{fig:colfig8}
\end{figure}

In the language of distance trees introduced at the beginning of Subsection \ref{subsec:dgraphs}, we are looking for certain hyperbolic embeddings of \emph{star trees} (or simply \emph{stars}).
Part (b) is, once again, a clear consequence of part (a).
We will use a similar outline and the same ideas employed in the proofs of Theorems \ref{thm:simpldensity} and \ref{thm:boxdensity} to establish Theorem \ref{thm:spacetimedensity} too.
See Figure \ref{fig:colfig8} for an illustration of the case $n=3$. 

Note that, in this hyperbolic setting, we are not allowing the scaling of the hyperbola $\mathscr{H}$ by some factor $\lambda$.
It is a non-compact curve, so it leaves more freedom than the circle to find non-scaled copies of the desired configuration.
For this reason, as it was discussed in \cite{BMK17}, the result analogous to Theorem \ref{thm:spacetimedensity} is possible for every indefinite quadratic form in place of $Q(x,y)=x^2-y^2$, but not for semi-definite ones.

\begin{remark}\label{rem:onstars}
As a consequence of Theorem \ref{thm:spacetimedensity}(a), one immediately obtains the same property of sets $A\subseteq\R^2$ of positive upper Banach density.
Clearly, it is not enough for $A$ to merely have infinite measure, as is seen from the trivial example: $n=1$, $a_1=1$, and $A=[-1/3,1/3]\times\R$.
\end{remark}

To avoid possible confusion, let us emphasize that this subsection did not study configurations determined by fixing some hyperbolic distances---it was ``hyperbolic'' rather in the sense of the Minkowski spacetime. Adapting our techniques to hyperbolic geometry would be a different task that we do not attempt in the present paper.


\subsection{Paper outline and the main ideas}
Each of Sections \ref{sec:R2}--\ref{sec:topes} and \ref{sec:proofsimpl}--\ref{sec:prooftrees} is dedicated to the proof of a single theorem. 
We first establish the ``negative'' results, since they are more elementary.
Theorem \ref{thm:colorR2} relies on an explicit partition of the complex plane $\mathbb{C}$, coloring the number $z$ based on the relative location of $z^2$ to the (scaled) Gaussian integers $\mathbb{Z}+\ii\mathbb{Z}$.
The construction could be thought of as a complex modification of the approach of Erd\H{o}s at al.\@ \cite[\S3]{Eetal1} who used $|z|^2$ in place of $z^2$.
The proof of Theorem \ref{thm:negparallel} uses the same idea, only with roughly $1/\sin\varepsilon$ many colors, while the proof of Theorem \ref{thm:hypercubes} extends it to higher complex powers, namely $z^n$ for $n\geq3$.
The proof of Theorem \ref{thm:colorRn} proceeds differently, by constructing the coloring that prevents monochromatic boxes that are only slightly rotated from their ``standard'' position and then using compactness of the group of rotations $\textup{SO}(n)$. 
One could say that the main novelty here is the construction of an almost invariant quantity for the boxes of the aforementioned type and not insisting on finding the exact invariant, which might not even exist in full generality.
The trick of using compactness of the underlying transformation group is essentially due to Straus~\cite{Str75}, but in a somewhat different setting. Also, its applicability is far from automatic: for instance, it can be used for rectangles but not for general parallelograms.
Finally, Theorem \ref{thm:parallel} has a rather straightforward inductive proof, by estimating the determinant representing the volume of a parallelotope.

Proofs of the ``positive'' density results rely on harmonic analysis, so a short Section \ref{sec:harm} serves as an interlude, containing a few basic facts and some notation.
The proof of Theorem \ref{thm:boxdensity} follows a similar outline as \cite[\S7]{DK22} by Durcik and the author, which, in turn, stems from a general outline of performing a certain regularity decomposition of the counting form \cite{B86:roth,Buk08,CMP15:roth}.
Similar proofs also appeared in a couple of other recent papers \cite{K20:anisotrop,KP23}, but here we also capitalize on the fact that the ``fixed volume'' condition is more flexible than fixing edge-lengths, so a lower ambient dimension suffices for that reason. 
The proof of Theorem \ref{thm:simpldensity} brings an additional novelty, by reinventing the notion of configuration size and defining it according to the lengths of some edges only, leaving room for pigeonholing and achieving the unit volume by fixing the length of the remaining edge. 
In the proof of Theorem \ref{thm:spacetimedensity} we need to take care of several edges simultaneously, making it also similar to \cite[\S3]{K20:anisotrop}.
Here we are making use of the fact that the underlying configuration is no longer rigid but flexible, which gives us additional useful ``degrees of freedom'' to detect the configuration analytically.
 
Section \ref{sec:remarksanalysts} contains a short philosophical discussion on why area $1$ triangles allow a density theorem (in sufficiently large dimensions), while area $1$ rectangles do not. 
We draw analogies with singular integral forms from multilinear harmonic analysis.
Section \ref{sec:summary}, the closing one, contains a simplified table summarizing both known and open results for several basic types of point configurations. 

The author hopes that the present paper will revive the topic of studying existence theorems for fixed-volume point configurations, which seems to have been explored in only a few papers like \cite{Mau02,DJ10,AR12,AC14,Mor15,KP25,Koi25} since the early work of Graham \cite{Gra80}.


\section{Rectangles: proof of Theorem \ref{thm:colorR2}}
\label{sec:R2}
We are about to give a coloring of $\mathbb{R}^2$ that uses $25$ colors and has a slightly stronger property: no color-class will contain the vertices of a parallelogram such that the product of lengths of its two consecutive sides equals $1$. For rectangles this clearly specializes to the property of their area being equal to $1$. 

\begin{figure}
\includegraphics[width=0.6\linewidth]{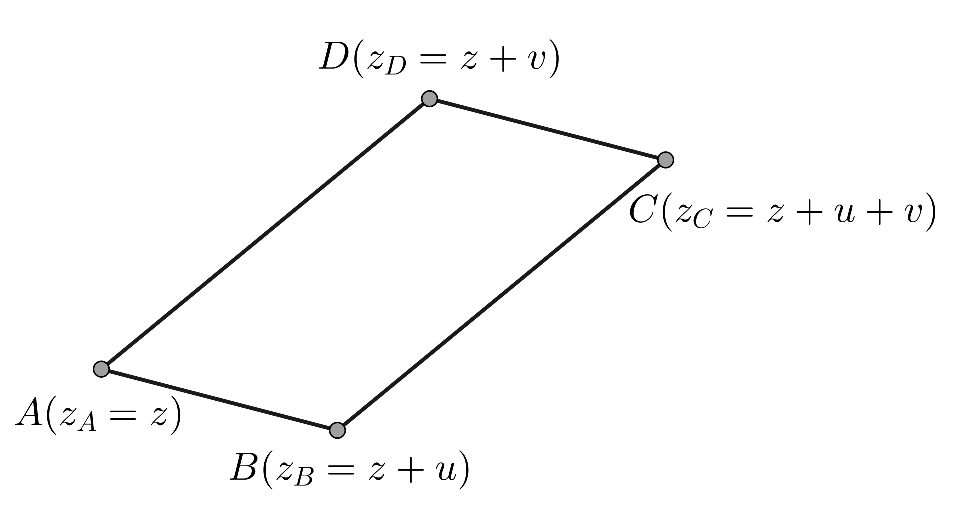}
\caption{Coordinatization of a parallelogram.}
\label{fig:colfig1}
\end{figure}

\begin{proof}[Proof of Theorem \ref{thm:colorR2}]
Let us place a (possibly degenerate) parallelogram $\mathcal{P}=ABCD$ in the complex plane, so that its vertices $A,B,C,D$ are respectively coordinatized by the complex numbers $z_A,z_B,z_C,z_D$ as in Figure \ref{fig:colfig1}. Consider a complex quantity $\mathscr{I}(\mathcal{P})$ defined as
\begin{equation}\label{eq:defofI}
\mathscr{I}(\mathcal{P}) := z_A^2 - z_B^2 + z_C^2 - z_D^2.
\end{equation}
In this definition we specify the vertex $A$ to be the one with the smallest coordinate $z_A$ in the lexicographic ordering of $\mathbb{C}\equiv\R^2$. Otherwise, $\mathscr{I}(\mathcal{P})$ would have only been determined up to multiplication by $\pm1$.

There exist $u,v,z\in\mathbb{C}$ such that the vertices of $\mathcal{P}$ have complex coordinates
\[ z_A=z, \quad z_B=z+u, \quad z_C=z+u+v, \quad z_D=z+v; \]
see Figure \ref{fig:colfig1} again.
The quantity $\mathscr{I}(\mathcal{P})$ now simplifies as 
\[ \mathscr{I}(\mathcal{P}) = z^2 - (z+u)^2 + (z+u+v)^2 - (z+v)^2 = 2uv. \]
Consecutive side-lengths of $\mathcal{P}$ are $|u|$ and $|v|$, so we have
\[ |\mathscr{I}(\mathcal{P})| = 2 \]
whenever their product equals $1$. As we have already mentioned, this holds in particular if $\mathcal{P}$ is a rectangle of area $1$.
Therefore, it remains to find a coloring of $\mathbb{C}$ such that, if all vertices of $\mathcal{P}$ are assigned the same color, then the complex number $\mathscr{I}(\mathcal{P})$ does not lie on the circle
\begin{equation}\label{eq:circle}
\{w\in\mathbb{C} : |w|=2\}.
\end{equation}

\begin{figure}
\includegraphics[width=0.6\linewidth]{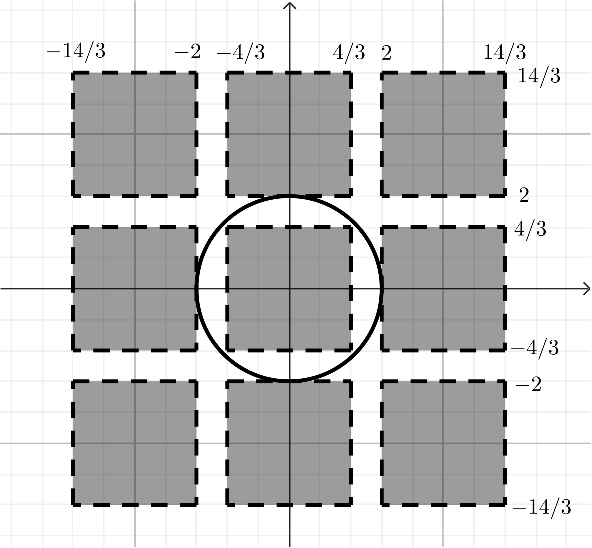}
\caption{The circle misses the squares.}
\label{fig:colfig2}
\end{figure}

For each pair $(j,k)\in\{0,1,2,3,4\}^2$ define a color-class $\mathscr{C}_{j,k}$ as
\[ \mathscr{C}_{j,k} := \bigg\{ z\in\mathbb{C} \,:\, z^2 \in \frac{10}{3} \bigg( \Z + \ii \Z + \frac{j + \ii k}{5} + \Big[0,\frac{1}{5}\Big) + \ii \Big[0,\frac{1}{5}\Big) \bigg) \bigg\}. \]
If the four vertices of $\mathcal{P}=ABCD$ belonged to the same color-class, then, by the definition \eqref{eq:defofI}, we would clearly have
\[ \mathscr{I}(\mathcal{P}) \in \frac{10}{3} \bigg( \Z + \ii \Z + \Big(-\frac{2}{5},\frac{2}{5}\Big) + \ii \Big(-\frac{2}{5},\frac{2}{5}\Big) \bigg). \]
The above set does not intersect the circle \eqref{eq:circle}; see Figure \ref{fig:colfig2}.
Indeed, the central square lies fully inside \eqref{eq:circle} because of $4\sqrt{2}/3<2$, while all remaining open squares clearly belong to its exterior. 
\end{proof}

We can say that the above solution relies on the invariant quantity $|\mathscr{I}(\mathcal{P})|$ assigned to rectangles of area $1$.
It does not generalize to all higher dimensions, since it uses multiplication of complex numbers, so we will resort to an ``almost invariant'' quantity in the next section. 

Let us illustrate the coloring constructed in the previous proof. Boundaries of the color-classes are given in the $(x,y)$-coordinate system by the equations
\[ x^2-y^2 = \frac{2a}{3} \quad\text{and}\quad xy = \frac{b}{3} \]
for arbitrary $a,b\in\Z$.
These are two mutually orthogonal families of hyperbolas (including degenerate ones for $a=0$ or $b=0$), depicted in Figure \ref{fig:colfig3}.

\begin{figure}
\includegraphics[width=0.5\linewidth]{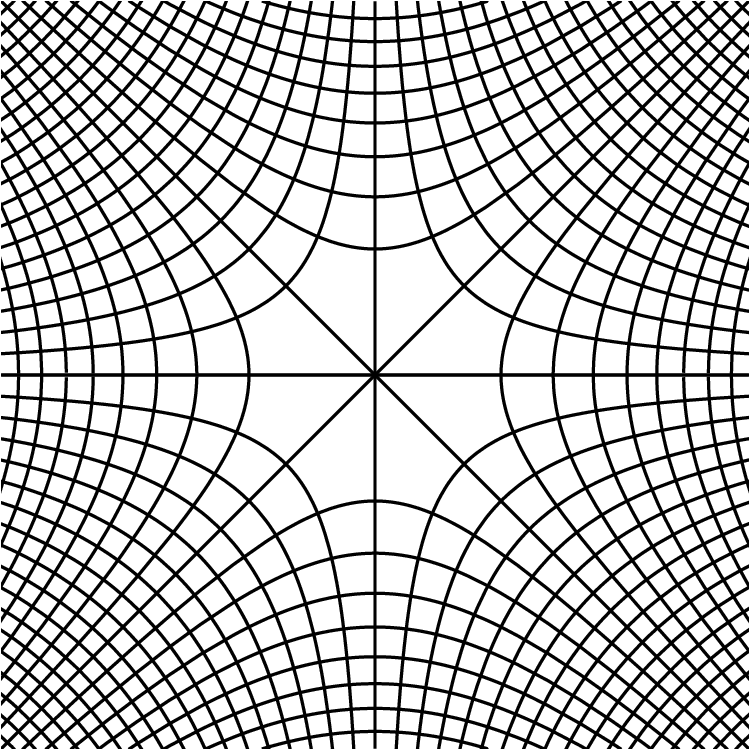}
\caption{Boundaries of color-classes $\mathscr{C}_{j,k}$.}
\label{fig:colfig3}
\end{figure}

The above proof of Theorem \ref{thm:colorR2} has recently been formalized in Lean \cite[\#189]{EP} with the help of Aristotle AI (by Harmonic) and Gemini 3 Pro (by Google), as part of joint efforts of the \emph{Erd\H{o}s problems} community \cite{EP}.


\section{Rectangular boxes: proof of Theorem \ref{thm:colorRn}}
\label{sec:Rn}
The idea is based on the following simple observation. Let us first take an $m$-dimensional rectangular box in $\mathbb{R}^n$ with edge-lengths $a_1,a_2,\ldots,a_m\in(0,\infty)$ and with sides parallel to the first $m$ coordinate axes of $\mathbb{R}^n$. Its vertices can then be enumerated by subsets $T$ of $\{1,2,\ldots,m\}$ as
\begin{equation}\label{eq:verticesofR0}
\Big( {q} + \sum_{j\in T} a_j \mathbbm{e}_j : T\subseteq\{1,2,\ldots,m\} \Big),
\end{equation}
where ${q}=(q_1,\ldots,q_n)\in\mathbb{R}^n$ is some point, namely, one of its vertices. Let us compute the alternating sum of the product of the first $m$ coordinates of the box vertices,
\[ \sum_{T\subseteq\{1,2,\ldots,m\}} (-1)^{m-|T|} \Big(\prod_{j\in T^c} q_j\Big) \Big(\prod_{j\in T} (q_j+a_j)\Big) = \prod_{j=1}^{m} (-q_j + q_j+a_j) = a_1 a_2 \cdots a_m, \]
and notice that we have obtained precisely the box volume.

The crucial part of the proof will be to show that the same quantity is almost invariant for slightly rotated rectangular boxes, where the ``slightness'' can be prescribed uniformly over all box eccentricities.
Note that this property is not quite obvious and it is sensitive to the pattern shape, as the uniformity will fail already for parallelograms in $\mathbb{R}^2$.
After we construct a coloring that prohibits those slightly tilted boxes, in the last step we will rotate it by finitely many matrices $U_1,\ldots,U_L\in\textup{SO}(n)$, thanks to compactness of the rotation group. 

Before the proof we need a simple identity.

\begin{lemma}\label{lm:identity}
The identity 
\begin{equation}\label{eq:identity}
\sum_{T\subseteq\{1,2,\ldots,m\}} (-1)^{m-|T|} \prod_{k=1}^{m} \Big(p_k + \sum_{j\in T} v_{j,k}\Big) = \sum_{\sigma\in \mathbb{S}_m} v_{1,\sigma(1)} v_{2,\sigma(2)} \cdots v_{m,\sigma(m)} 
\end{equation}
holds for complex numbers $(p_k)_{1\leq k\leq m}$ and $(v_{j,k})_{1\leq j,k\leq m}$.
\end{lemma}

Here $\mathbb{S}_m$ denotes the group of permutations on the set $\{1,2,\ldots,m\}$.
It is peculiar to notice that the right hand side ($\textup{RHS}$) of \eqref{eq:identity} is the permanent of the matrix $(v_{j,k})_{j,k}$.
In fact, identity \eqref{eq:identity} is a generalization of Ryser's formula for the permanent \cite[Corollary 4.2]{Rys63}, which is obtained by specializing $p_k=0$ for every $k$. We will give a different self-contained proof.

\begin{proof}[Proof of Lemma \ref{lm:identity}]
We will prove the identity by induction on $m$.
The basis case $m=1$ is trivial as then the identity reads
\[ - p_1 + (p_1 + v_{1,1}) = v_{1,1}. \]
Take a positive integer $m\geq2$.
Observe that the left hand side ($\textup{LHS}$) of \eqref{eq:identity} is a homogeneous polynomial of degree $m$ in $m^2+m$ variables $p_k$ and $v_{j,k}$, but the degree of each of the variables in it is at most $1$.
Differentiating it with respect to the variable $v_{m,l}$ for some $l\in\{1,\ldots,m\}$ we obtain
\begin{align*}
\frac{\partial}{\partial v_{m,l}} \textup{LHS}
& = \sum_{T\subseteq\{1,\ldots,m-1\}} (-1)^{m-1-|T|} \prod_{\substack{1\leq k\leq m\\k\neq l}} \Big(p_k + v_{m,k} + \sum_{j\in T} v_{j,k}\Big) \\
& = \sum_{\substack{\sigma\in \mathbb{S}_m\\\sigma(m)=l}} v_{1,\sigma(1)} v_{2,\sigma(2)} \cdots v_{m-1,\sigma(m-1)},
\end{align*}
where in the last equality we applied the induction hypothesis to a smaller collection of numbers
\[ (p_k + v_{m,k})_{1\leq k\leq m, k\neq l},\quad (v_{j,k})_{1\leq j\leq m-1,1\leq k\leq m,k\neq l} \]
and renamed the indices.
Multiplying by $v_{m,l}$ and summing in $l$ gives
\[ \sum_{l=1}^{m} v_{m,l} \frac{\partial}{\partial v_{m,l}} \textup{LHS} = \textup{RHS}. \]
Observe that the polynomial 
\[ \textup{LHS} - \sum_{l=1}^{m} v_{m,l} \frac{\partial}{\partial v_{m,l}} \textup{LHS} \]
does not depend on the variables $v_{m,1},\ldots,v_{m,m}$, because they appear with power at most $1$ and thus certainly cancel when evaluating the above difference.
Therefore, the quantity $\textup{LHS} - \textup{RHS}$ also does not depend on $v_{m,1},\ldots,v_{m,m}$, so it can be evaluated simply by plugging $v_{m,1}=\cdots=v_{m,m}=0$:
\[ \textup{LHS} - \textup{RHS} = \sum_{T\subseteq\{1,\ldots,m\}} (-1)^{m-|T|} \prod_{k=1}^{m} \Big(p_k + \sum_{\substack{1\leq j\leq m-1\\j\in T}} v_{j,k}\Big) = 0. \]
The last sum equals $0$ because its terms can be paired to cancel each other: adding/subtracting the element $m$ to/from the set $T$ gives exactly the same term, only with the opposite signs.
We have obtained $\textup{LHS} = \textup{RHS}$.
\end{proof}

Now we turn to the proof of the announced result.

\begin{proof}[Proof of Theorem \ref{thm:colorRn}]
It is sufficient to fix integers $m\leq n$, since later we can consider a joint refinement (i.e., mutual class-wise intersections) of the constructed colorings of $\R^n$ obtained for $m=1,2,\ldots,n$.
To each $m$-dimensional rectangular box $\mathcal{R}$ in $\R^n$ we assign a real-valued quantity $\mathscr{J}(\mathcal{R})$ defined as
\begin{equation}\label{eq:defofJ}
\mathscr{J}(\mathcal{R}) := \sum_{{x}=(x_1,\ldots,x_n) \text{ is a vertex of }\mathcal{R}} (-1)^{m-\text{par}({x})} \,x_1\cdots x_m. 
\end{equation}
Note that only the first $m$ out of $n$ coordinates of a point ${x}$ appear in the product $x_1\cdots x_m$.
Also, here $\text{par}({x})$ denotes the \emph{parity} of a vertex ${x}$, which is computed as follows.
Choose the base vertex of $\mathcal{R}$ to be the one with the smallest coordinate representation in the lexicographic ordering of $\R^n$. The parity of any vertex of $\mathcal{R}$ is defined to be the parity of its distance from the base vertex in the $1$-skeleton graph of $\mathcal{R}$. This number is either $0$ or $1$ and it changes as we move from a vertex to its neighbor along a $1$-edge of $\mathcal{R}$.
Clearly, the expression \eqref{eq:defofJ} has $2^m$ terms, half of them get the $+$ sign and half of them come with the $-$ sign; see the illustration of these signs in Figure \ref{fig:colfig4}.

\begin{figure}
\includegraphics[width=0.65\linewidth]{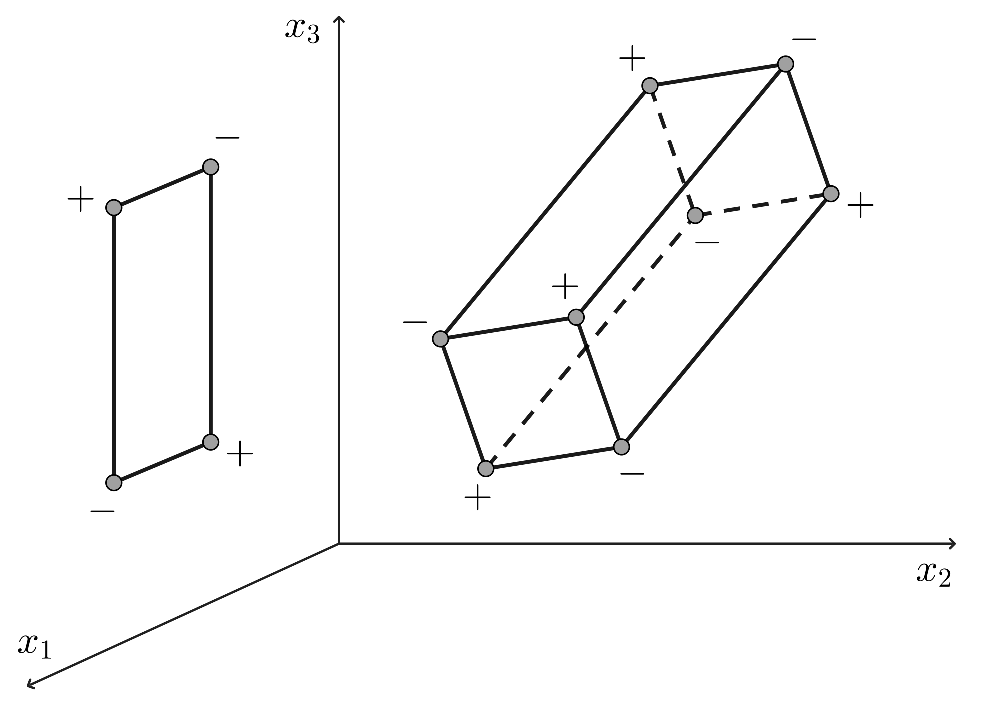}
\caption{Signs attached to rectangle and box vertices (cases $m=2$ and $m=3$).}
\label{fig:colfig4}
\end{figure}

Every $m$-dimensional rectangular box $\mathcal{R}\subseteq\mathbb{R}^n$ can be obtained by rotating an axes-aligned box $\mathcal{R}_0$ about the origin, 
\begin{equation}\label{eq:recimage}
\mathcal{R} = U \mathcal{R}_0,
\end{equation}
where the vertices of $\mathcal{R}_0$ are given by \eqref{eq:verticesofR0}, while the rotation $U$ is given by a special orthogonal transformation, $U\in\textup{SO}(n)$.
The vertices of $\mathcal{R}$ can then be written as
\[ \Big( {p} + \sum_{j\in T} {v}_j : T\subseteq\{1,2,\ldots,m\} \Big), \]
where ${p}=U{q}$ and ${v}_j=a_j U\mathbbm{e}_j$ for each index $1\leq j\leq m$. 
Also, write coordinatewise:
\begin{equation}\label{eq:pvcoordinates}
{p} = (p_k)_{1\leq k\leq n}, \quad {v}_j = (v_{j,k})_{1\leq k\leq n}
\end{equation}
and note that Lemma \ref{lm:identity} gives
\begin{equation}\label{eq:repofJ}
\mathscr{J}(\mathcal{R}) = \pm \sum_{\sigma\in \mathbb{S}_m} v_{1,\sigma(1)} v_{2,\sigma(2)} \cdots v_{m,\sigma(m)}.
\end{equation}

First, suppose that the rotation $U$ satisfies $\|U-I\|_{\textup{op}}<\varepsilon$, where $\|T\|_{\textup{op}}$ denotes the operator norm of a matrix $T$ and
\[ \varepsilon := \frac{1}{2^{m+2}m!}. \]
In particular,
\[ \Big| \frac{1}{a_j}{v}_j - \mathbbm{e}_j \Big| = |(U-I)\mathbbm{e}_j| < \varepsilon, \]
so that
\[ | v_{j,j} - a_j | < \varepsilon a_j \]
for $1\leq j\leq m$ and 
\[ | v_{j,k} | < \varepsilon a_j \]
for $1\leq j\leq m$, $1\leq k\leq n$, $j\neq k$.
Consequently,
\begin{align}
& \Big| \sum_{\sigma\in \mathbb{S}_m} v_{1,\sigma(1)} v_{2,\sigma(2)} \cdots v_{m,\sigma(m)} - a_1 a_2 \cdots a_m \Big| \nonumber \\
& \leq | v_{1,1} v_{2,2} \cdots v_{m,m} - a_1 a_2 \cdots a_m |
+ \sum_{\substack{\sigma\in \mathbb{S}_m\\ \sigma\neq\textup{id}}} | v_{1,\sigma(1)} v_{2,\sigma(2)} \cdots v_{m,\sigma(m)} | \nonumber \\
& < m \varepsilon (1+\varepsilon)^{m-1} a_1 a_2 \cdots a_m + (m!-1) \varepsilon (1+\varepsilon)^{m-1} a_1 a_2 \cdots a_m \nonumber \\
& \leq 2^m m! \varepsilon a_1 a_2 \cdots a_m \leq \frac{1}{4} a_1 a_2 \cdots a_m. \label{repofJaux}
\end{align}
Let us now additionally assume that $\mathcal{R}$ has volume $1$. Then \eqref{eq:repofJ} combined with \eqref{repofJaux} and $a_1 a_2 \cdots a_m=1$ gives
\begin{equation}\label{eq:Jisinint1}
\mathscr{J}(\mathcal{R}) \in \Big(-\frac{5}{4},-\frac{3}{4}\Big) \cup \Big(\frac{3}{4},\frac{5}{4}\Big). 
\end{equation}
Partition $\mathbb{R}^n$ into the sets
\[ \mathscr{S}_l := \bigg\{ (x_1,x_2,\ldots,x_n) \in \R^n : x_1 x_2 \cdots x_m \in \frac{3}{2}\bigg(\mathbb{Z} + \Big[\frac{l}{3\cdot2^{m}},\frac{l+1}{3\cdot2^{m}}\Big)\bigg) \bigg\} \]
for $0\leq l\leq 3\cdot2^{m}-1$.
We claim that the vertices of $\mathcal{R}$ cannot all belong to the same set $\mathscr{S}_l$. Namely, if they did, then the definition of $\mathscr{J}(\mathcal{R})$ would give
\begin{equation}\label{eq:Jisinint2}
\mathscr{J}(\mathcal{R}) \in \frac{3}{2}\mathbb{Z} + \Big(-\frac{1}{4},\frac{1}{4}\Big)
= \cdots \cup \Big(-\frac{7}{4},-\frac{5}{4}\Big) \cup \Big(-\frac{1}{4},\frac{1}{4}\Big) \cup \Big(\frac{5}{4},\frac{7}{4}\Big) \cup \cdots. 
\end{equation}
However, \eqref{eq:Jisinint1} and \eqref{eq:Jisinint2} together lead to a contradiction as the sets on their right hand sides are disjoint.

Finally, we handle completely arbitrary rectangular boxes $\mathcal{R}$ of unit volume.
Consider an open neighborhood $\mathcal{O}$ of the identity $I$ in the rotation group $\textup{SO}(n)$ defined as
\[ \mathcal{O} := \{ V\in\textup{SO}(n) \,:\, \|V-I\|_{\textup{op}}<\varepsilon \}. \]
Then the family $\{U\mathcal{O} : U\in\textup{SO}(n)\}$ constitutes an open cover of the compact space $\textup{SO}(n)$, so it can be reduced to a finite subcover $\{U_1\mathcal{O}, U_2\mathcal{O}, \ldots, U_L\mathcal{O}\}$.
The color-classes of the desired coloring of $\R^n$ can now be defined as
\[ \mathscr{C}_{l_1,l_2,\ldots,l_L} := (U_1 \mathscr{S}_{l_1}) \cap (U_2 \mathscr{S}_{l_2}) \cap \cdots \cap (U_L \mathscr{S}_{l_L}), \]
where $(l_1,l_2,\ldots,l_L)$ run over all $L$-tuples of elements from $\{0,1,2,\ldots,3\cdot 2^m-1\}$.
In words, this is a joint refinement of $L$ mutually rotated colorings.
Suppose that the vertices of $\mathcal{R}$ belong to the same color-class $\mathscr{C}_{l_1,l_2,\ldots,l_L}$. 
Let $U\in\textup{SO}(n)$ be as in \eqref{eq:recimage}, but without any assumption on the norm of $U-I$. 
Take an index $i\in\{1,\ldots,L\}$ such that $U\in U_i\mathcal{O}$. Then the box 
\[ \mathcal{R}' := U_i^{-1} \mathcal{R} \]
satisfies
\[ \mathcal{R}' = U_i^{-1} U \mathcal{R}_0, \quad \|U_i^{-1} U-I\|_{\textup{op}}<\varepsilon \]
and all of its vertices are in the set
\[ U_i^{-1} \mathscr{C}_{l_1,l_2,\ldots,l_L} \subseteq U_i^{-1} U_i \mathscr{S}_{l_i} = \mathscr{S}_{l_i}. \]
This contradicts the construction from the previous part of the proof.
\end{proof}

The number of colors needed in the above proof grows superexponentially in $n$, as remarked in the introduction. This follows from the known growth of the minimal cardinality $L(n)$ of a cover of $\textup{SO}(n)$ by translates of the set $\mathcal{O}$ (or of any other fixed small neighborhood of $I$).
Namely, when the Haar measure $\mu_{\textup{SO}(n)}$ of the orthogonal group is normalized to have total mass $1$, then the large deviation principle for the spectral measure of the orthogonal group estimates $\mu_{\textup{SO}(n)}(\mathcal{O}) = O((\varepsilon/2)^{n^2/2})$ as $n\to\infty$ for a fixed $\varepsilon\in(0,1)$; we omit the lengthy details.
Therefore, it remains to apply the ``volume bound''
\[ L(n) \geq \frac{\mu_{\textup{SO}(n)}(\textup{SO}(n))}{\mu_{\textup{SO}(n)}(\mathcal{O})} \]
to conclude the superexponential growth of $L(n)$ in $n$.


\section{Parallelograms: proof of Theorem \ref{thm:negparallel}}
This will be an easy modification of the proof of Theorem \ref{thm:colorR2}.
Thus, we work in the complex plane and keep the notation from Section \ref{sec:R2}.

\begin{proof}[Proof of Theorem \ref{thm:negparallel}]
Fix $\varepsilon\in(0,\pi/2)$, take a positive integer $N$ greater than $6/\sin\varepsilon$, and this time define the color-classes as
\[ \mathscr{C}'_{j,k} := \bigg\{ z\in\mathbb{C} \,:\, z^2 \in \frac{4}{\sin\varepsilon} \bigg( \Z + \ii \Z + \frac{j + \ii k}{N} + \Big[0,\frac{1}{N}\Big) + \ii \Big[0,\frac{1}{N}\Big) \bigg) \bigg\} \]
for integers $0\leq j,k\leq N-1$.
For every monochromatic parallelogram $\mathcal{P}$ coordinatized as in Section \ref{sec:R2} (see Figure \ref{fig:colfig1}) we conclude
\[ u v = \frac{1}{2}\mathscr{I}(\mathcal{P}) \in \frac{2}{\sin\varepsilon} \bigg( \Z + \ii \Z + \Big(-\frac{2}{N},\frac{2}{N}\Big) + \ii \Big(-\frac{2}{N},\frac{2}{N}\Big) \bigg) \]
and, in particular,
\[ |u v| \not\in \Big[1,\frac{1}{\sin\varepsilon}\Big]. \]
If all angles of $\mathcal{P}$ are at least $\varepsilon$, then its area is comparable to the product of its side-lengths. More precisely,
\[ |uv| \sin\varepsilon \leq |\mathcal{P}| \leq |uv| \]
and we see that $\mathop{\textup{area}}(\mathcal{P})$ cannot equal $1$.

Next, we can partition $\mathbb{C}$ into
\[ \mathscr{C}''_{j} := \bigg\{ z\in\mathbb{C} \,:\, \mathop{\textup{Im}}(z^2) \in 4 \bigg( \Z + \frac{j}{5} + \Big[0,\frac{1}{5}\Big) \bigg) \bigg\} \]
for $0\leq j\leq 4$.
For every parallelogram $\mathcal{P}$ that is monochromatic with respect to this coloring we certainly have
\begin{align*}
\mathop{\textup{Im}}(uv) & = \frac{1}{2} \big(\mathop{\textup{Im}}(z^2) - \mathop{\textup{Im}}((z+u)^2) + \mathop{\textup{Im}}((z+u+v)^2) - \mathop{\textup{Im}}((z+v)^2) \big) \\
& \in 2\Z + \Big(-\frac{4}{5},\frac{4}{5}\Big).
\end{align*}
If one side of $\mathcal{P}$ is parallel to the real axis, then $u\in\R$ or $v\in\R$, so the area of $\mathcal{P}$ is
\[ |\mathcal{P}| = |\mathop{\textup{Im}}(uv)| \]
and we again conclude that this area cannot be equal to $1$. 
This is not surprising because we have in fact been working with the quantity $\mathscr{J}$ from \eqref{eq:defofJ} in Section \ref{sec:Rn} and we have just shown that it is invariant, but only for these special parallelograms.
Rotate the coloring $\mathscr{C}''_0,\ldots,\mathscr{C}''_4$ in such a way that it prevents monochromatic area $1$ parallelograms from having one side parallel to one of the given lines $\ell_i$; repeat this for $i=1,2,\ldots,m$.

It remains to take a joint refinement of all colorings of $\mathbb{C}$ constructed in this proof.
\end{proof}


\section{Hypercube graphs: proof of Theorem \ref{thm:hypercubes}}
Let us conveniently work in the complex plane again.
The $1$-skeleton of \eqref{eq:nbox} is then embedded in $\mathbb{C}$ as in Figure \ref{fig:colfig5} given in the introductory section.

\begin{proof}[Proof of Theorem \ref{thm:hypercubes}]
By taking $p_k=z$ and $v_{j,k}=u_j$ for each index $k$ in the identity \eqref{eq:identity} in Lemma \ref{lm:identity} from Section \ref{sec:Rn}, we obtain a simpler (and well-known) polynomial identity
\begin{equation}\label{eq:complplyiden}
\sum_{T\subseteq\{1,2,\ldots,n\}} (-1)^{n-|T|} \Big(z + \sum_{j\in T} u_j\Big)^n = n! \,u_1 u_2 \cdots u_n 
\end{equation}
for $z,u_1,u_2,\ldots,u_n\in\mathbb{C}$.
Now we partition $\mathbb{C}$ into color-classes defined as
\[ \mathscr{C}_{j,k} := \bigg\{ z\in\mathbb{C} \,:\, z^n \in 2 n! \bigg( \Z + \ii \Z + \frac{j + \ii k}{2^{n+1}} + \Big[0,\frac{1}{2^{n+1}}\Big) + \ii \Big[0,\frac{1}{2^{n+1}}\Big) \bigg) \bigg\} \]
for $0\leq j,k\leq 2^{n+1}-1$.
If all $2^n$ points \eqref{eq:1skeleton} belong to the same class $\mathscr{C}_{j,k}$, then \eqref{eq:complplyiden} implies
\[ u_1 u_2 \cdots u_n \in 2(\Z + \ii \Z) + \Big(-\frac{1}{2},\frac{1}{2}\Big) + \ii\Big(-\frac{1}{2},\frac{1}{2}\Big). \]
From this it is clear that $a_1\cdots a_n=|u_1\cdots u_n|$ cannot be equal to $1$.
\end{proof}


\section{Parallelotopes: proof of Theorem \ref{thm:parallel}}
\label{sec:topes}
\begin{proof}[Proof of Theorem \ref{thm:parallel}]
By induction on a positive integer $n$ we will prove that every $n$-parallelotope with vertices in the set
\[ A_{n,\theta} := \{ (x_1,x_2,\ldots,x_n)\in(0,\infty)^n : x_1 x_2\cdots x_n \leq \theta \}, \]
for some $\theta>0$, has $n$-volume strictly smaller than $n!\theta$.
See Figure \ref{fig:colfig6} for an illustration of the case $n=2$.
Afterwards, it will remain to choose $\theta=\varepsilon/n!$ and observe that $A_{n,\theta}$ has volume
\[ |A_{n,\theta}| = \int_{(0,\infty)^{n-1}} \frac{\theta \,\textup{d}x_1 \cdots \textup{d}x_{n-1}}{x_1 \cdots x_{n-1}} = \infty \]
for every $n\geq2$.

\begin{figure}
\includegraphics[width=0.85\linewidth]{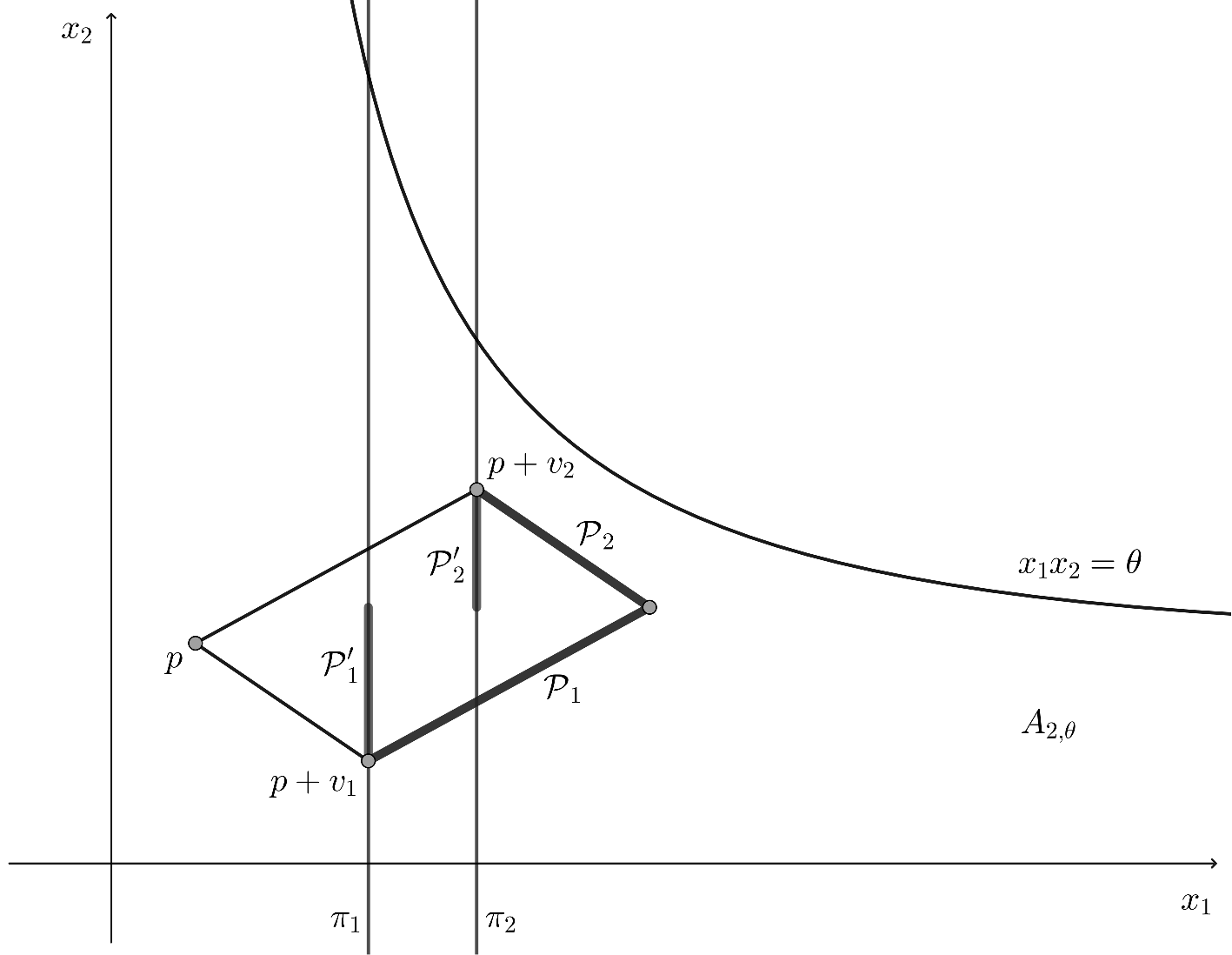}
\caption{A parallelogram with vertices in $A_{2,\theta}$.}
\label{fig:colfig6}
\end{figure}

The induction basis $n=1$ is trivial since the observed $1$-parallelotopes are just segments in the interval $(0,\theta]$, so their length is clearly less than $\theta$.
Now take $n\geq2$ and suppose that the claim holds for the sets $A_{n-1,\vartheta}$ for every $\vartheta>0$.
Let $\theta>0$ and let $\mathcal{P}$ be an $n$-parallelotope with vertices in $A_{n,\theta}$. 
Suppose that ${p}$ is the vertex of $\mathcal{P}$ with the smallest first coordinate (i.e., coordinate $x_1$). Also suppose that $\mathcal{P}$ is spanned from ${p}$ by vectors ${v}_1,\ldots,{v}_n$  
and introduce their coordinates just as in formula \eqref{eq:pvcoordinates}.
We have 
\begin{equation}\label{eq:positivecoordinates}
v_{j,1}\geq0 \quad \text{for} \quad 1\leq j\leq n
\end{equation}
by the choice of the vertex ${p}$.
Finally, for each $1\leq i\leq n$, let $\pi_i$ denote the hyperplane orthogonal to the first coordinate axis and passing through the point ${p}+{v}_i$ (i.e., with Cartesian equation $x_1=p_1+v_{i,1}$),
let $\mathcal{P}_i$ denote the $(n-1)$-parallelotope spanned from the point ${p}+{v}_i$ by the vectors ${v}_1,\ldots,{v}_{i-1},{v}_{i+1},\ldots,{v}_n$, and let $\mathcal{P}'_i$ be the orthogonal projection of $\mathcal{P}_i$ onto $\pi_i$; see Figure \ref{fig:colfig6}.
From \eqref{eq:positivecoordinates} we see that all vertices of $\mathcal{P}_i$ have $x_1$-coordinates at least $p_1+v_{i,1}$, so the vertices of $\mathcal{P}'_i$ lie in the set
\[ A_{n,\theta} \cap \pi_i = \{p_1+v_{i,1}\} \times A_{n-1,\theta/(p_1+v_{i,1})}. \]
By the induction hypothesis applied inside $\pi_i\cong\R^{n-1}$, its volume satisfies
\begin{equation}\label{eq:boundedprime}
|\mathcal{P}'_i| < \frac{(n-1)! \theta}{p_1+v_{i,1}}.
\end{equation}
Using the Laplace expansion of a determinant (with respect to the first column), \eqref{eq:positivecoordinates}, and \eqref{eq:boundedprime} we can finally estimate the volume of $\mathcal{P}$ as
\begin{align*} 
|\mathcal{P}| & = |\det(v_{j,k})_{1\leq j,k\leq n}| 
\leq \sum_{i=1}^{n} v_{i,1} \Big|\det(v_{j,k})_{\substack{1\leq j\leq n,\,j\neq i\\ 2\leq k\leq n}}\Big| \\
& = \sum_{i=1}^{n} v_{i,1} |\mathcal{P}'_i|
\leq \sum_{i=1}^{n} v_{i,1}\frac{(n-1)! \theta}{p_1+v_{i,1}} < n! \theta,
\end{align*}
which concludes the induction step and completes the proof.
\end{proof}


\section{Notation and preliminaries from harmonic analysis}
\label{sec:harm}
From this section on, we introduce the notation
\[ \mathcal{A} \lesssim_P \mathcal{B} \quad\text{and}\quad \mathcal{B} \gtrsim_P \mathcal{A}, \]
which means that inequality $\mathcal{A} \leq C_P \mathcal{B}$ holds with some constant $C_P\in(0,\infty)$ depending only on a set of parameters $P$.
It is always understood that inequality constants depend on the (configuration or ambient) dimensions, so these are not explicitly listed in the set $P$. 
Next, 
\[ \mathcal{A} \sim_P \mathcal{B} \]
means that both $\mathcal{A} \lesssim_P \mathcal{B}$ and $\mathcal{B} \lesssim_P \mathcal{A}$ hold simultaneously.

When we have a function $f$ of two (or more) variables mapping $(x,y)\mapsto f(x,y)$, then we write $f(\cdot,y)$ for the single-variable function that maps $x\mapsto f(x,y)$ for a fixed $y$. Similarly, $f(x,\cdot)$ denotes a single-variable function that assigns $y\mapsto f(x,y)$.

The standard Gaussian function on $\R^d$ is
\[ \mathbbm{g}(x) := e^{-\pi|x|^2}, \]
its partial derivatives are
\[ \mathbbm{h}^{(l)}(x) := \frac{\partial}{\partial x_l} \mathbbm{g}(x) = -2\pi x_l e^{-\pi|x|^2} \]
for $1\leq l\leq d$, while its Laplacian will be denoted as
\[ \mathbbm{k}(x) := \Delta\mathbbm{g}(x) = (4\pi^2|x|^2 - 2d\pi)e^{-\pi|x|^2}. \]
Dimension $d$ is not visible from the notation $\mathbbm{g},\mathbbm{h}^{(l)},\mathbbm{k}$, but it will always be understood from context.

If the \emph{Fourier transformation} operator $f\mapsto\widehat{f}$ is normalized such that
\[ \widehat{f}(\xi) := \int_{\R^d} f(x) e^{-2\pi\ii x\cdot\xi} \,\textup{d}x \]
for every Schwartz function $f$ on $\R^d$, then the Fourier transforms of the above Gaussian functions are given as:
\begin{equation}\label{eq:Ftransforms} 
\widehat{\mathbbm{g}}(\xi) = e^{-\pi |\xi|^2},\quad 
\widehat{\mathbbm{h}^{(l)}}(\xi) = 2\pi\mathbbm{i}\xi_l \,e^{-\pi |\xi|^2},\quad 
\widehat{\mathbbm{k}}(\xi) = - 4 \pi^2 |\xi|^2 e^{-\pi |\xi|^2}.
\end{equation}
Also the Fourier transform of a finite positive (or even complex) Borel measure on $\R^d$ is defined as
\[ \widehat{\mu}(\xi) := \int_{\R^d} e^{-2\pi\ii x\cdot\xi} \,\textup{d}\mu(x). \]
The \emph{convolution} of $\mu$ and $f$ is a new function $\mu\ast f$ given by the formula
\[ (\mu\ast f)(x) := \int_{\R^d} f(x-y) \,\textup{d}\mu(y) \]
and the property 
\[ \widehat{\mu\ast f} = \widehat{\mu} \widehat{f} \]
is well-known to hold.
For every $\alpha\in(0,\infty)$ the scaled function $f_\alpha$ and the scaled measure $\mu_\alpha$ are defined by
\[ f_\alpha(x) := \alpha^{-d} f(\alpha^{-1}x),\quad \mu_\alpha(E) := \mu(\alpha^{-1}E) \]
and their Fourier transforms are computed as
\[ \widehat{f_\alpha}(\xi) = \widehat{f}(\alpha\xi),\quad \widehat{\mu_\alpha}(\xi) = \widehat{\mu}(\alpha\xi). \]
All of these basic properties of Fourier transforms are found in every textbook on real harmonic analysis, such as \cite{SW71}.

Anisotropic scalings are also sometimes useful. We define them for functions $f$ and measures $\mu$ on $\R^2$ for every pair $(\alpha,\beta)\in(0,\infty)^2$ as
\[ f_{\alpha,\beta}(x,y) := \alpha^{-1}\beta^{-1} f(\alpha^{-1}x, \beta^{-1}y), \quad  \mu_{\alpha,\beta}(E) := \mu(\{(\alpha^{-1}x,\beta^{-1}y):(x,y)\in E\}) \]
and then we have
\[ \widehat{f_{\alpha,\beta}}(\zeta,\eta) = \widehat{f}(\alpha\zeta,\beta\eta),\quad \widehat{\mu_{\alpha,\beta}}(\zeta,\eta) = \widehat{\mu}(\alpha\zeta,\beta\eta). \]
Note that all of the functions $f_\alpha,f_{\alpha,\beta}$ have the same $\textup{L}^1$-norm as $f$, while the measures $\mu_\alpha,\mu_{\alpha,\beta}$ have the same total mass as $\mu$. 

Immediate consequences of \eqref{eq:Ftransforms} are the Gaussian convolution identities:
\begin{align} 
\sum_{l=1}^{d} \mathbbm{h}^{(l)}_{\alpha} \ast \mathbbm{h}^{(l)}_{\beta} & = \frac{\alpha\beta}{\alpha^2+\beta^2} \mathbbm{k}_{\sqrt{\alpha^2+\beta^2}}, \label{eq:convidhh} \\
\mathbbm{k}_{\alpha} \ast \mathbbm{g}_{\beta} & = \frac{\alpha^2}{\alpha^2+\beta^2} \mathbbm{k}_{\sqrt{\alpha^2+\beta^2}} \label{eq:convidkg}
\end{align}
for $\alpha,\beta\in(0,\infty)$.
Similarly, in two dimensions,
\begin{align}
\sum_{l=1}^{2} \mathbbm{h}^{(l)}_{\alpha,\beta} \ast \mathbbm{h}^{(l)}_{\alpha,\beta} & = \frac{1}{2} \mathbbm{k}_{\alpha\sqrt{2},\beta\sqrt{2}} \label{eq:convidaa} \\
\mathbbm{k}_{\alpha,\beta} \ast \mathbbm{g}_{\alpha,\beta} & = \frac{1}{2} \mathbbm{k}_{\alpha\sqrt{2},\beta\sqrt{2}} \label{eq:convidab}
\end{align}
for $\alpha,\beta\in(0,\infty)$.
If $\mu$ is a compactly supported finite positive Borel measure on $\R^d$, then
\begin{equation}\label{eq:convlowerbd}
\mu\ast\mathbbm{g} \gtrsim_{\mu} \mathbbm{1}_{[-1,1]^{d}}.
\end{equation}
It is allowed that the implicit constant in this pointwise inequality depends on the choice of the measure $\mu$. 
In fact, it can be taken to be the minimum of $\mathbbm{g}$ on the set
\[ \{x-y \,:\, x\in[-1,1]^{d},\ y\in\mathop{\textup{supp}}\mu\} \]
multiplied by the total mass of $\mu$. 
In particular, this holds when $\mu=\delta_0$ is the Dirac measure at the origin, when it reads $\mathbbm{g} \gtrsim \mathbbm{1}_{[-1,1]^{d}}$, which is quite obvious anyway. Slightly more generally, we also have
\begin{equation}\label{eq:lowerbd2}
\mathbbm{g}_{\alpha} \gtrsim_P \mathbbm{1}_{[-1,1]^{d}}.
\end{equation}
when $\alpha\sim_P 1$.
When $\mu$ is a smooth truncation of the surface measure of a hyper-surface with non-vanishing Gaussian curvature, then its Fourier transform decays as 
\begin{equation}\label{eq:Fdecay}
\big|\widehat{\mu}(\xi)\big| \lesssim_\mu (1+|\xi|)^{-(d-1)/2};
\end{equation}
see \cite[\S{}VIII.3]{St93:book}.
In particular, this holds when $\mu=\sigma$ is the spherical measure supported on the standard unit sphere $\mathbb{S}^{d-1}\subseteq\R^d$, also called the circle measure when $d=2$.
Note that an easy consequence of \eqref{eq:Fdecay} and \eqref{eq:Ftransforms} is
\begin{equation}\label{eq:measurewithk}
\sup_{\xi\in\R^d} \big|\widehat{\mu}(\xi)\widehat{\mathbbm{k}}(t\xi)\big| 
\lesssim_\mu \sup_{\xi\in\R^d} |\xi|^{-1/2} \big|\widehat{\mathbbm{k}}(t\xi)\big|
\sim t^{1/2}
\end{equation}
when $d\geq2$.
The \emph{heat equation} in the present reparametrization of the time $t$ reads
\begin{equation}\label{eq:heatequ} 
\frac{\partial}{\partial t} \mathbbm{g}_t(x) = \frac{1}{2\pi t} \mathbbm{k}_t(x) 
\end{equation}
for $(t,x)\in (0,\infty)\times\mathbb{R}^d$. It can be verified by direct differentiation.

For a measurable function $f\colon\R^d\to\mathbb{C}$ and $p\in[1,\infty)$ we denote (the so-called \emph{Lebesgue seminorm})
\[ \|f\|_{\textup{L}^p(\R^d)} := \Big( \int_{\R^d} |f(x)|^p \,\textup{d}x \Big)^{1/p}. \]
The endpoint case $p=\infty$ is interpreted as the essential supremum of the image of $|f|$,
\[ \|f\|_{\textup{L}^\infty(\R^d)} := \mathop{\textup{ess\,sup}}_{x\in\R^d} |f(x)|. \]
The standard $\textup{L}^2$ inner product is 
\[ \langle f,g\rangle_{\textup{L}^2(\R^d)} := \int_{\R^d} f(x) \overline{g(x)} \,\textup{d}x . \]
\emph{Young's convolution inequality} reads
\begin{equation}\label{eq:Youngineq}
\|f\ast g\|_{\textup{L}^r(\R^d)} \leq \|f\|_{\textup{L}^p(\R^d)} \|g\|_{\textup{L}^q(\R^d)}
\end{equation}
when the exponents $p,q,r\in[1,\infty]$ satisfy $1+1/r=1/p+1/q$; see \cite[\S{}V.1]{SW71}


\section{Right simplices: proofs of Theorem \ref{thm:simpldensity} and Corollary \ref{cor:simpldensity}}
\label{sec:proofsimpl}
We only need to prove part (a) of Theorem \ref{thm:simpldensity}.
It is sufficient to work in dimension $n=m+1$, since, in higher dimensions $n$, Fubini's theorem guarantees that at least one of the sections parallel to a fixed $(m+1)$-dimensional coordinate plane has density at least that of $A$.
The parameter $m$ is fixed from now on. Take an $R\in(0,\infty)$ and a measurable set $A\subseteq[0,R]^{m+1}$ with density
\[ \delta: = |A|/R^{m+1} \in (0,1]. \] 
Furthermore, denote
\begin{equation}\label{eq:simplchoiceoftheta}
\theta = m^{-1} 2^{-m^2-m-1} \delta^{m+1}
\end{equation}
throughout the proof.
For a parameter $\lambda\in(0,\infty)$, interpreted as a certain ``scale'' that we consider one at a time, we define a \emph{configuration-counting form} as
\begin{align*}
\mathcal{N}_{\lambda}^{0}(A;R) := & \int_{\R^{m-1}} \int_{\R^{2}} \int_{\R^{m-1}} \int_{\R^{2}} \mathbbm{1}_A(x,y) \Big(\prod_{k=1}^{m-1}\mathbbm{1}_A(x+u_k \mathbbm{e}_k, y)\Big) \mathbbm{1}_A(x,y+v) \\
& \textup{d}\sigma_{m!|u_1\cdots u_{m-1}|^{-1}}(v) \,\lambda^{-m+1} \Big(\prod_{k=1}^{m-1}\mathbbm{1}_{[-\lambda,-\theta\lambda]\cup[\theta\lambda,\lambda]}(u_k)\Big) \,\textup{d}u \,R^{-m-1} \,\textup{d}y \,\textup{d}x ,
\end{align*}
where $\sigma$ is the circle measure in $\R^2$ normalized to have total mass $1$ on $\mathbb{S}^1$ and we write
\[ u=(u_1,\ldots,u_{m-1}) \in \mathbb{R}^{m-1} \]
for shortness.
Observe that $\mathcal{N}_{\lambda}^{0}(A;R)$ is a certain density of a subcollection of axes-aligned right-simplices inside $A$: it detects when all of the points
\begin{equation}\label{eq:wefoundsimplex}
(x,y),\ (x+u_1\mathbbm{e}_1,y),\ldots\ (x+u_{m-1}\mathbbm{e}_{m-1},y),\ (x,y+v) 
\end{equation}
belong to $A$ by performing normalized integration over all simplex vertices $(x,y)\in [0,R]^{m+1}$, over $m-1$ simplex edges parallel to the coordinate axes with lengths $\theta\lambda\leq|u_k|\leq\lambda$, and over the last perpendicular edge parallel to the remaining coordinate $2$-plane with its length precisely equal to
$|v|=m!|u_1\cdots u_{m-1}|^{-1}$,
so that the $m$-dimensional simplex volume equals $1$. Merely knowing that $\mathcal{N}_{\lambda}^{0}(A;R)>0$ guarantees existence of the desired point configuration inside $A$. 
The freedom to choose the parameter $\lambda$ will allow certain pigeonholing.
The idea is to show that $\mathcal{N}_{\lambda}^{0}$ is positive for at least one out of sufficiently many choices for $\lambda$, i.e., from a sufficiently large interval on the logarithmic scale, which is available as soon as $R$ is large enough (in terms of $\delta$).

The main idea, traced back to Bourgain \cite{B86:roth} and emphasized by Cook, Magyar, and Pramanik \cite{CMP15:roth}, is to consider a \emph{certain smoothed counting form}, which we here define for every $\varepsilon\in(0,1]$ by convolving the circle measure $\sigma$ with a scaled Gaussian $\mathbbm{g}_\varepsilon$:
\begin{align*}
\mathcal{N}_{\lambda}^{\varepsilon}(A;R) := & \int_{\R^{m-1}} \int_{\R^{2}} \int_{\R^{m-1}} \int_{\R^{2}} \mathbbm{1}_A(x,y) \Big(\prod_{k=1}^{m-1}\mathbbm{1}_A(x+u_k \mathbbm{e}_k, y)\Big) \mathbbm{1}_A(x,y+v) \\
& \times (\sigma \ast \mathbbm{g}_{\varepsilon})_{m!|u_1\cdots u_{m-1}|^{-1}}(v)
\,\lambda^{-m+1} \Big(\prod_{k=1}^{m-1}\mathbbm{1}_{[-\lambda,-\theta\lambda]\cup[\theta\lambda,\lambda]}(u_k)\Big)
\,\textup{d}v \,\textup{d}u \,R^{-m-1} \,\textup{d}y \,\textup{d}x .
\end{align*}
It approximates the original counting form in the limit when $\varepsilon\to0$, as is shown by the following simple lemma.

\begin{lemma}\label{lm:simpconv}
For every $R,\lambda\in(0,\infty)$ and a measurable set $A\subseteq[0,R]^{m+1}$ we have
\[ \lim_{\varepsilon\to0+} \mathcal{N}_{\lambda}^{\varepsilon}(A;R) = \mathcal{N}_{\lambda}^{0}(A;R). \]
\end{lemma}

\begin{proof}[Proof of Lemma \ref{lm:simpconv}]
Denoting
\[ F(u,v) := \int_{\R^{m-1}} \int_{\R^{2}} \mathbbm{1}_A(x,y) \Big(\prod_{k=1}^{m-1}\mathbbm{1}_A(x+u_k \mathbbm{e}_k, y)\Big) \mathbbm{1}_A(x,y+v) \,\textup{d}y \,\textup{d}x \]
and expanding out the convolution in the integral
\[ \int_{\R^{2}} F(u,v) \,\big(\sigma_{m!|u_1\cdots u_{m-1}|^{-1}} \ast \mathbbm{g}_{\varepsilon m!|u_1\cdots u_{m-1}|^{-1}}\big)(v) \,\textup{d}v , \]
we see that $\mathcal{N}_{\lambda}^{\varepsilon}(A;R)$ can be rewritten as
\begin{align*}
\mathcal{N}_{\lambda}^{\varepsilon}(A;R) = \int_{\R^{m-1}} \int_{\R^{2}} \big( F(u,\cdot) \ast \mathbbm{g}_{\varepsilon m!|u_1\cdots u_{m-1}|^{-1}} \big) \,\textup{d}\sigma_{m!|u_1\cdots u_{m-1}|^{-1}} & \\
\times R^{-m-1}\lambda^{-m+1} \Big(\prod_{k=1}^{m-1}\mathbbm{1}_{[-\lambda,-\theta\lambda]\cup[\theta\lambda,\lambda]}(u_k)\Big) \,\textup{d}u & .
\end{align*}
Observe that $F\colon\R^{m-1}\times\R^{2}\to[0,\infty)$ is bounded and continuous. The dominated convergence theorem, applied first to the inner integral for each fixed $u$ and then to the outer integral in $u$, proves that $\mathcal{N}_{\lambda}^{\varepsilon}(A;R)$ converges as $\varepsilon\to0$ to
\begin{align*}
\int_{\R^{m-1}} \int_{\R^{2}} F(u,\cdot) \,\textup{d}\sigma_{m!|u_1\cdots u_{m-1}|^{-1}}
\,R^{-m-1}\lambda^{-m+1} \Big(\prod_{k=1}^{m-1}\mathbbm{1}_{[-\lambda,-\theta\lambda]\cup[\theta\lambda,\lambda]}(u_k)\Big) \,\textup{d}u & ,
\end{align*}
which is precisely $\mathcal{N}_{\lambda}^{0}(A;R)$.
\end{proof}

The strategy is to decompose:
\begin{equation}\label{eq:decomposition}
\mathcal{N}_{\lambda}^{0}(A;R) = \underbrace{\mathcal{N}_{\lambda}^{1}(A;R)}_{\text{structured part}} + \big(\underbrace{\mathcal{N}_{\lambda}^{\varepsilon}(A;R)-\mathcal{N}_{\lambda}^{1}(A;R)}_{\text{error part}}\big) + \big(\underbrace{\mathcal{N}_{\lambda}^{0}(A;R)-\mathcal{N}_{\lambda}^{\varepsilon}(A;R)}_{\text{uniform part}}\big), 
\end{equation}
where the names of the \emph{structured}, \emph{error}, and \emph{uniform} parts are borrowed from regularity lemmas in extremal combinatorics, which they mimic. The structured part is the dominant and positive one, the uniform part is always small, while the error part can be made small only after successful pigeonholing in sufficiently many choices of $\lambda$; see a general discussion of the approach in \cite[\S1.3]{K20:anisotrop}.

In line with that, as the first ingredient we show that the structured part is bounded from below only in terms of $\delta$, uniformly in the parameter $\lambda$ from a certain range.

\begin{lemma}\label{lm:simplstruct}
For every $R,\lambda\in(0,\infty)$ such that $R^{-1/(m-1)}\leq\lambda\leq R$ and a measurable set $A\subseteq[0,R]^{m+1}$ of density $\delta=|A|/R^{m+1}$ we have
\[ \mathcal{N}_{\lambda}^{1}(A;R) \gtrsim \delta^{(m+1)(2m-1)}. \]
\end{lemma}

Note that the implicit constant depends on $m$ alone.

\begin{proof}[Proof of Lemma \ref{lm:simplstruct}]
Recall the bound \eqref{eq:convlowerbd} for the convolution $\sigma\ast\mathbbm{g}$, which can be scaled to give
\begin{align*}
(\sigma\ast\mathbbm{g})_{m!|u_1\cdots u_{m-1}|^{-1}} & \gtrsim |u_1\cdots u_{m-1}|^{2} \mathbbm{1}_{[-m!|u_1\cdots u_{m-1}|^{-1},m!|u_1\cdots u_{m-1}|^{-1}]^{2}} \\
& \geq (\theta \lambda)^{2(m-1)} \mathbbm{1}_{[-\lambda^{-m+1},\lambda^{-m+1}]^{2}}
\end{align*}
for $u_1,\ldots,u_{m-1}\in[\theta\lambda,\lambda]$.
Consequently,
\begin{equation}\label{eq:simpllowerN1}
\mathcal{N}_{\lambda}^{1}(A;R) \gtrsim \theta^{2(m-1)} R^{-m-1} \lambda^{m-1} \int_{D\setminus E} \Phi(u,v,x,y) \,\textup{d}(u,v,x,y),
\end{equation}
where
\begin{align*}
D := \big\{ (u,v,x,y) :\ & |u_k|\leq\lambda\text{ for every }1\leq k\leq m-1, \\
& |v_k|\leq \lambda^{-m+1}\text{ for every }1\leq k\leq 2, \\ 
& 0\leq x_k\leq R\text{ for every }1\leq k\leq m-1, \\ 
& 0\leq y_k\leq R\text{ for every }1\leq k\leq 2 \big\},
\end{align*}
\[ E := \big\{ (u,v,x,y)\in D : |u_k|<\theta\lambda\text{ for some }1\leq k\leq m-1 \big\}, \]
and we also denoted
\[ \Phi(u,v,x,y) := \mathbbm{1}_A(x,y) \Big(\prod_{k=1}^{m-1}\mathbbm{1}_A(x+u_k \mathbbm{e}_k, y)\Big) \mathbbm{1}_A(x,y+v). \]
Let $\mathcal{I}$ denote a partition of $[0,\lceil R/\lambda\rceil\lambda]\supseteq[0,R]$ into intervals of length $\lambda$ and let $\mathcal{Q}$ be a partition of $[0,\lceil R\lambda^{m-1}\rceil\lambda^{-m+1}]^{2}\supseteq[0,R]^{2}$ into squares of side-length $\lambda^{-m+1}$. Clearly, by the construction and the assumption $R\geq\max\{\lambda,\lambda^{-m+1}\}$,
\begin{equation}\label{eq:simplcard}
\mathop{\textup{card}}\mathcal{I} = \Big\lceil \frac{R}{\lambda}\Big\rceil \leq \frac{2R}{\lambda}, \quad \mathop{\textup{card}}\mathcal{Q} = \lceil R\lambda^{m-1}\rceil^{2} \leq (2R\lambda^{m-1})^{2}. 
\end{equation}

The integral
\begin{equation}\label{eq:uinteg1}
\int_{D} \Phi(u,v,x,y) \,\textup{d}(u,v,x,y) 
\end{equation}
is greater than or equal to the sum of integrals
\begin{align}
\mathcal{A} := \int_{I_1\times\cdots\times I_{m-1}\times Q} \Big(\prod_{k=1}^{m-1}\int_{I_k}\mathbbm{1}_A(x_1,\ldots,x_{k-1},x',x_{k+1},\ldots,x_{m-1}, y) \,\textup{d}x'\Big) & \nonumber \\
\times \Big(\int_Q \mathbbm{1}_A(x,y') \,\textup{d}y\Big) \mathbbm{1}_A(x,y) \,\textup{d}(x,y) & \label{eq:uinteg2}
\end{align}
over all choices of $I_1,\ldots,I_{m-1}\in\mathcal{I}$ and $Q\in\mathcal{Q}$.
The total number of such choices is, by \eqref{eq:simplcard},
\begin{equation}\label{eq:simplcompare}
N := (\mathop{\textup{card}}\mathcal{I})^{m-1} (\mathop{\textup{card}}\mathcal{Q})
\leq (2R)^{m+1} \lambda^{m-1}.
\end{equation}
Observe the equalities
\begin{align*}
\mathcal{B} & := \prod_{k=1}^{m-1}\int_{I_1\times\cdots\times I_{m-1}\times Q}\frac{\mathbbm{1}_A(x,y)}{\int_{I_k}\mathbbm{1}_A(\ldots,x_{k-1},x',x_{k+1},\ldots, y) \,\textup{d}x'} \,\textup{d}(x,y) \\
& \,= \big(\lambda^{m-2}(\lambda^{-m+1})^{2}\big)^{m-1} 
= \lambda^{-(m-1)m}
\end{align*}
and
\[ \mathcal{C} := \int_{I_1\times\cdots\times I_{m-1}\times Q} \frac{\mathbbm{1}_A(x,y)}{\int_Q\mathbbm{1}_A(x,y')\,\textup{d}y'} \,\textup{d}(x,y) = \lambda^{m-1}, \]
where the fraction $0/0$ is interpreted as $0$.
We now apply H\"older's inequality for integration over $I_1\times\cdots\times I_{m-1}\times Q$, with $m+1$ factors and all exponents equal $m+1$. It gives
\[ \mathcal{A}\, \mathcal{B}\, \mathcal{C} \geq \Big( \int_{I_1\times\cdots\times I_{m-1}\times Q} \mathbbm{1}_A(x,y)^{m+1} \,\textup{d}(x,y) \Big)^{m+1}, \]
so the integral $\mathcal{A}$ from \eqref{eq:uinteg2} is at least
\[ \lambda^{(m-1)^2} |A\cap(I_1\times\cdots\times I_{m-1}\times Q)|^{m+1}. \]
Thus, \eqref{eq:uinteg1} is bounded from below by
\begin{align*}
& \lambda^{(m-1)^2} \sum_{\substack{I_1,\ldots,I_{m-1}\in\mathcal{I}\\Q\in\mathcal{Q}}} |A\cap(I_1\times\cdots\times I_{m-1}\times Q)|^{m+1} \\
& \geq \lambda^{(m-1)^2} N
\Big( \frac{1}{N} \sum_{\substack{I_1,\ldots,I_{m-1}\in\mathcal{I}\\Q\in\mathcal{Q}}} |A\cap(I_1\times\cdots\times I_{m-1}\times Q)| \Big)^{m+1} \\
& = \lambda^{(m-1)^2} N^{-m} |A|^{m+1} \\
& \geq (2R)^{-m(m+1)} \lambda^{-m+1} |A|^{m+1},
\end{align*}
by Jensen's inequality and \eqref{eq:simplcompare}.
Finally, using the assumption $|A|=\delta R^{m+1}$, we get
\begin{equation}\label{eq:simplintD}
\int_{D} \Phi \geq 2^{-m(m+1)} \delta^{m+1} R^{m+1} \lambda^{-m+1}.
\end{equation}

Next, the integral
\begin{equation*}
\int_{E} \Phi(u,v,x,y) \,\textup{d}(u,v,x,y) 
\end{equation*}
needs to be bounded from above. On the part of the set $E$ where $|u_k|<\theta \lambda$ for a fixed given index $1\leq k\leq m-1$ we simply estimate $\Phi$ by $1$ and integrate out in all variables, bounding that part of the integral by
\[ (\theta\lambda) \lambda^{m-2} (\lambda^{-m+1})^{2} R^{m-1} R^{2} = \theta \lambda^{-m+1} R^{m+1}. \]
Therefore, 
\begin{equation}\label{eq:simplintE}
\int_{E} \Phi \leq m \theta R^{m+1} \lambda^{-m+1}.
\end{equation}
From \eqref{eq:simpllowerN1}, \eqref{eq:simplintD}, and \eqref{eq:simplintE} we conclude
\[ \mathcal{N}_{\lambda}^{1}(A;R) \gtrsim \theta^{2(m-1)} (2^{-m(m+1)}\delta^{m+1} - m\theta), \]
so choosing $\theta$ as in \eqref{eq:simplchoiceoftheta} we end up with
\[ \mathcal{N}_{\lambda}^{1}(A;R) \gtrsim \delta^{2(m+1)(m-1)+m+1} = \delta^{(m+1)(2m-1)}, \]
as desired.
\end{proof}

Next, we turn to the uniform part, for which we show that it decays as $\varepsilon\to0$ uniformly in $\lambda$.

\begin{lemma}\label{lm:simplunif}
For every $\lambda,R\in(0,\infty)$, $\varepsilon\in(0,1]$, and a measurable set $A\subseteq[0,R]^{m+1}$ we have
\[ \big|\mathcal{N}_{\lambda}^{0}(A;R)-\mathcal{N}_{\lambda}^{\varepsilon}(A;R)\big| \lesssim \varepsilon^{1/2}. \]
\end{lemma}

\begin{proof}[Proof of Lemma \ref{lm:simplunif}]
In this proof we conveniently denote
\[ f(x,y;u) := \mathbbm{1}_A(x,y) \prod_{k=1}^{m-1}\mathbbm{1}_A(x+u_k \mathbbm{e}_k, y) \]
and only remember that $f(\cdot;u)$ is bounded by $1$ and supported on $[0,R]^{m+1}$. 
Also define
\[ a = a(u_1,\ldots,u_{m-1}) \]
to be an abbreviation for
\[ a := m!|u_1\cdots u_{m-1}|^{-1}. \]
For any $0<\tau<\varepsilon$, by using the heat equation \eqref{eq:heatequ} in the integral form
\begin{equation}\label{eq:heatequ2}
(\sigma \ast \mathbbm{g}_{\tau})(v) - (\sigma \ast \mathbbm{g}_{\varepsilon})(v) = - \int_{\tau}^{\varepsilon} (\sigma \ast \mathbbm{k}_{t})(v) \,\frac{\textup{d}t}{2\pi t} ,
\end{equation}
we can write
\begin{align*}
\mathcal{N}_{\lambda}^{\tau}(A;R)-\mathcal{N}_{\lambda}^{\varepsilon}(A;R) = 
- \int_{\tau}^{\varepsilon} \int_{\R^{m-1}} \int_{\R^{m-1}} \int_{\R^{2}} \int_{\R^{2}} f(x,y;u) \mathbbm{1}_A(x,y+v) R^{-m-1} \lambda^{-m+1} & \\
\times (\sigma \ast \mathbbm{k}_{t})_{a}(v)
\Big(\prod_{k=1}^{m-1}\mathbbm{1}_{[-\lambda,-\theta\lambda]\cup[\theta\lambda,\lambda]}(u_k)\Big)
\,\textup{d}v \,\textup{d}y \,\textup{d}u \,\textup{d}x \,\frac{\textup{d}t}{2\pi t} & .
\end{align*}
If $\widehat{f}(x,\xi;u)$ denotes the Fourier transform of $y\mapsto f(x,y;u)$ evaluated at $\xi$ and if we similarly define $\widehat{\mathbbm{1}_A}(x,\xi)$, then we can rewrite 
\[ \int_{\R^{2}} \int_{\R^{2}} f(x,y;u) \mathbbm{1}_A(x,y+v) (\sigma \ast \mathbbm{k}_{t})_{a}(v) \,\textup{d}v \,\textup{d}y \]
using Plancherel's theorem on the frequency side as
\[ \int_{\R^{2}} \widehat{\mathbbm{1}_A}(x,\xi) \,\overline{\widehat{f}(x,\xi;u)} \,\widehat{\sigma}(a\xi) \,\widehat{\mathbbm{k}}(t a\xi) \,\textup{d}\xi. \]
This is, by \eqref{eq:measurewithk}, the Cauchy-Schwarz inequality, and another application of Plancherel's theorem, bounded in the absolute value by a constant multiple of
\begin{align*}
& t^{1/2} \Big( \int_{\R^{2}} \big|\widehat{\mathbbm{1}_A}(x,\xi)\big|^2 \,\textup{d}\xi \Big)^{1/2}
\Big( \int_{\R^{2}} \big|\widehat{f}(x,\xi;u)\big|^2 \,\textup{d}\xi \Big)^{1/2} \\
& = t^{1/2} \Big( \int_{\R^{2}} \mathbbm{1}_A(x,y)^2 \,\textup{d}y \Big)^{1/2}
\Big( \int_{\R^{2}} f(x,y;u)^2 \,\textup{d}y \Big)^{1/2}
\leq t^{1/2} R^{2}.
\end{align*}
Thus,
\begin{align*}
\big|\mathcal{N}_{\lambda}^{\tau}(A;R)-\mathcal{N}_{\lambda}^{\varepsilon}(A;R)\big|
\lesssim \int_{\tau}^{\varepsilon} \int_{[0,R]^{m-1}} \int_{[-\lambda,\lambda]^{m-1}}
t^{1/2} R^{-m+1} \lambda^{-m+1} \,\textup{d}u \,\textup{d}x \,\frac{\textup{d}t}{t} & \\
\lesssim \int_{\tau}^{\varepsilon} t^{-1/2}\,\textup{d}t \lesssim \varepsilon^{1/2} & .
\end{align*}
The desired estimate is obtained in the limit as $\tau\to0$, thanks to Lemma \ref{lm:simpconv}.
\end{proof}

The error part is not controlled for a single scale $\lambda$, but rather ``on the average'' over all scales $\lambda=e^\alpha$, $\alpha\in\R$.

\begin{lemma}\label{lm:simplerr}
For every $R\in(0,\infty)$, $\varepsilon\in(0,1]$, and a measurable set $A\subseteq[0,R]^{m+1}$ we have
\[ \int_{\R} \big( \mathcal{N}_{e^{\alpha}}^{\varepsilon}(A;R)-\mathcal{N}_{e^{\alpha}}^{1}(A;R) \big)^2 \,\textup{d}\alpha \lesssim \theta^{-4(m-1)} \Big(\log\frac{1}{\varepsilon}\Big)^2 . \]
\end{lemma}

\begin{proof}[Proof of Lemma \ref{lm:simplerr}]
Using the same notation as in the previous proof and applying \eqref{eq:heatequ2} again, we write
\begin{align*}
\mathcal{N}_{\lambda}^{\varepsilon}(A;R)-\mathcal{N}_{\lambda}^{1}(A;R) = 
- \int_{\varepsilon}^{1} \int_{e^{-1}t\lambda^{-m+1}}^{t\lambda^{-m+1}} \int_{\R^{m-1}} \int_{\R^{m-1}} \int_{\R^{2}} \int_{\R^{2}} f(x,y;u) \mathbbm{1}_A(x,y+v) & \\
R^{-m-1} \lambda^{-m+1} (\sigma \ast \mathbbm{k}_{t})_{a}(v)
\Big(\prod_{k=1}^{m-1}\mathbbm{1}_{[-\lambda,-\theta\lambda]\cup[\theta\lambda,\lambda]}(u_k)\Big)
\,\textup{d}v \,\textup{d}y \,\textup{d}u \,\textup{d}x \,\frac{\textup{d}s}{s} \,\frac{\textup{d}t}{2\pi t} & .
\end{align*}
Here we also introduced the new integration variable,
\begin{equation}\label{eq:simplwheres} 
s\in[e^{-1}t\lambda^{-m+1},t\lambda^{-m+1}]. 
\end{equation}
The corresponding integral of $1/s$ equals $1$, so the above equality is still valid.
Whenever $\lambda\in(0,\infty)$, $t\in[\varepsilon,1]$, $s$ is as in \eqref{eq:simplwheres}, and 
$\theta\lambda\leq|u_k|\leq\lambda$
for $1\leq k\leq m-1$, then we define
\begin{equation}\label{eq:defofr1}
r = r(\lambda,t,s,u_1,\ldots,u_{m-1})\in(0,\infty)
\end{equation}
as the unique solution to
\begin{equation}\label{eq:defofr2}
2s^2 + r^2 = (ta)^2.
\end{equation}
It is well-defined since $ta\geq m!t\lambda^{-m+1}\geq 2s$.
Then convolution identities \eqref{eq:convidkg} and \eqref{eq:convidhh} yield
\[ \mathbbm{k}_{ta} = \Big(\frac{ta}{s}\Big)^2 \sum_{l=1}^{2} \mathbbm{h}^{(l)}_{s} \ast \mathbbm{h}^{(l)}_{s} \ast \mathbbm{g}_{r}, \]
so
\begin{align*} 
& - \int_{\R^{2}} \int_{\R^{2}} f(x,y;u) \mathbbm{1}(x,y+v) (\sigma \ast \mathbbm{k}_{t})_{a}(v) \,\textup{d}v \,\textup{d}y \\
& = \Big(\frac{ta}{s}\Big)^2 \sum_{l=1}^{2} \int_{\R^{2}} \int_{\R^{2}} \Big( \int_{\R^{2}} f(x,y;u) \mathbbm{h}^{(l)}_{s}(y-z) \,\textup{d}y \Big) \\
& \qquad\qquad\qquad\qquad \times\Big( \int_{\R^{2}} \mathbbm{1}_{A}(x,y') \mathbbm{h}^{(l)}_{s}(y'-z') \,\textup{d}y' \Big) (\sigma_{a}\ast\mathbbm{g}_{r})(z-z') \,\textup{d}z \,\textup{d}z'.
\end{align*}
If the numbers $t,s,u_k$ are as before, then
$ta/s \lesssim \theta^{-m+1}$, which in turn implies
\begin{align*}
\big| \mathcal{N}_{\lambda}^{\varepsilon}(A;R) - \mathcal{N}_{\lambda}^{1}(A;R) \big|
\lesssim\, & (\theta^2\lambda)^{-m+1}  R^{-m-1} \sum_{l=1}^{2} \int_{\varepsilon}^{1} \int_{e^{-1}t\lambda^{-m+1}}^{t\lambda^{-m+1}} \int_{\R^{m-1}} \int_{\R^{m-1}} \int_{\R^{2}} \int_{\R^{2}} \\
& \Big| \int_{\R^{2}} f(x,y;u) \mathbbm{h}^{(l)}_{s}(y-z) \,\textup{d}y \Big| 
\Big| \int_{\R^{2}} \mathbbm{1}_{A}(x,y') \mathbbm{h}^{(l)}_{s}(y'-z') \,\textup{d}y' \Big| \\
& \times (\sigma_a \ast \mathbbm{g}_{r})(z-z') \Big(\prod_{k=1}^{m-1}\mathbbm{1}_{[-\lambda,-\theta\lambda]\cup[\theta\lambda,\lambda]}(u_k)\Big) 
\,\textup{d}z \,\textup{d}z' \,\textup{d}u \,\textup{d}x \,\frac{\textup{d}s}{s} \,\frac{\textup{d}t}{t} .
\end{align*}
By the Cauchy--Schwarz inequality (in all variables but $y$ and $y'$) we estimate
\begin{equation}\label{eq:splitAB}
\big( \mathcal{N}_{\lambda}^{\varepsilon}(A;R)-\mathcal{N}_{\lambda}^{1}(A;R) \big)^2 \lesssim (\theta^4 \lambda^2)^{-m+1}  R^{-2m-2} \mathcal{A} \,\mathcal{B},
\end{equation}
where 
\begin{align*}
\mathcal{A} : = \sum_{l=1}^{2} \int_{\varepsilon}^{1} \int_{e^{-1}t\lambda^{-m+1}}^{t\lambda^{-m+1}} \int_{[0,R]^{m-1}} \int_{\R^{m-1}} \int_{\R^{2}} \int_{\R^{2}} 
\Big( \int_{\R^{2}} f(x,y;u) \mathbbm{h}^{(l)}_{s}(y-z) \,\textup{d}y \Big)^2 & \\
\times (\sigma_a \ast \mathbbm{g}_{r})(z-z') \Big(\prod_{k=1}^{m-1}\mathbbm{1}_{[-\lambda,-\theta\lambda]\cup[\theta\lambda,\lambda]}(u_k)\Big) 
\,\textup{d}z \,\textup{d}z' \,\textup{d}u \,\textup{d}x \,\frac{\textup{d}s}{s} \,\frac{\textup{d}t}{t} &
\end{align*}
and
\begin{align*}
\mathcal{B} : = \sum_{l=1}^{2} \int_{\varepsilon}^{1} \int_{e^{-1}t\lambda^{-m+1}}^{t\lambda^{-m+1}} \int_{[0,R]^{m-1}} \int_{\R^{m-1}} \int_{\R^{2}} \int_{\R^{2}} \Big( \int_{\R^{2}} \mathbbm{1}_{A}(x,y') \mathbbm{h}^{(l)}_{s}(y'-z') \,\textup{d}y' \Big)^2 & \\
\times (\sigma_a \ast \mathbbm{g}_{r})(z-z') \Big(\prod_{k=1}^{m-1}\mathbbm{1}_{[-\lambda,-\theta\lambda]\cup[\theta\lambda,\lambda]}(u_k)\Big) 
\,\textup{d}z \,\textup{d}z' \,\textup{d}u \,\textup{d}x \,\frac{\textup{d}s}{s} \,\frac{\textup{d}t}{t} & .
\end{align*}

First, we turn our attention to $\mathcal{A}$. Integrating out in $z'$, estimating crudely by Young's inequality \eqref{eq:Youngineq},
{\allowdisplaybreaks\begin{align*}
& \int_{\R^2} \Big( \int_{\R^{2}} f(x,y;u) \mathbbm{h}^{(l)}_{s}(y-z) \,\textup{d}y \Big)^2 \,\textup{d}z \\
& \leq \int_{\R^2} \int_{\R^2} f(x,y;u) f(x,y';u)\big(|\mathbbm{h}^{(l)}_{s}|\ast|\mathbbm{h}^{(l)}_{s}|\big)(y-y') \,\textup{d}y \,\textup{d}y' \\
& = \big\langle f(x,\cdot;u)\ast |\mathbbm{h}^{(l)}_{s}|\ast|\mathbbm{h}^{(l)}_{s}| , f(x,\cdot;u) \big\rangle_{\textup{L}^2(\R^2)} \\
& \leq \|f(x,\cdot;u)\|_{\textup{L}^2(\R^2)}^2 \|\mathbbm{h}^{(l)}\|_{\textup{L}^1(\R^2)}^2 \lesssim R^2,
\end{align*}}
and finally integrating out in $u,x,s,t$, we bound
\[ \mathcal{A} \lesssim R^{m+1} \lambda^{m-1} \log\frac{1}{\varepsilon}. \]
On the other hand, $\mathcal{B}$ is easily controlled by integrating immediately in $z$ and $u$:
\begin{align*} 
\mathcal{B} & \lesssim \lambda^{m-1} \sum_{l=1}^{2} \int_{\varepsilon}^{1} \int_{e^{-1}t\lambda^{-m+1}}^{t\lambda^{-m+1}} \int_{\R^{m-1}} \int_{\R^{2}} \Big( \int_{\R^{2}} \mathbbm{1}_{A}(x,y') \mathbbm{h}^{(l)}_{s}(y'-z') \,\textup{d}y' \Big)^2 \,\textup{d}z' \,\textup{d}x \,\frac{\textup{d}s}{s} \,\frac{\textup{d}t}{t} \\
& = -\frac{1}{2} \lambda^{m-1} \int_{\varepsilon}^{1} \int_{e^{-1}t\lambda^{-m+1}}^{t\lambda^{-m+1}} \int_{\R^{m-1}} \big\langle \mathbbm{1}_{A}(x,\cdot)\ast \mathbbm{k}_{s\sqrt{2}}, \mathbbm{1}_{A}(x,\cdot) \big\rangle_{\textup{L}^2(\R^2)} \,\textup{d}x \,\frac{\textup{d}s}{s} \,\frac{\textup{d}t}{t} \\
& = 4\pi^2 \lambda^{m-1} \int_{\varepsilon}^{1} \int_{e^{-1}t\lambda^{-m+1}}^{t\lambda^{-m+1}} \int_{\R^{m-1}} \int_{\R^2} \big|\widehat{\mathbbm{1}_{A}}(x,\xi)\big|^2 e^{-2\pi s^2 |\xi|^2} |\xi|^2 \,\textup{d}\xi \,\textup{d}x \,s\,\textup{d}s \,\frac{\textup{d}t}{t} ,
\end{align*}
where we also used \eqref{eq:convidhh}, Plancherel's theorem, and the third formula from \eqref{eq:Ftransforms}.
Combine the obtained estimates for $\mathcal{A}$ and $\mathcal{B}$ with \eqref{eq:splitAB}, substitute $\lambda=e^{\alpha}$, and integrate  in $\alpha$ to obtain
\begin{align*}
\int_{\R} \big( \mathcal{N}_{e^{\alpha}}^{\varepsilon}(A;R) - & \mathcal{N}_{e^{\alpha}}^{1}(A;R) \big)^2 \,\textup{d}\alpha 
\lesssim \theta^{-4m+4} R^{-m-1} \Big(\log\frac{1}{\varepsilon}\Big) \\
& \times \int_{\varepsilon}^{1} \int_{\R} \int_{te^{-(m-1)\alpha-1}}^{te^{-(m-1)\alpha}} \int_{\R^{m-1}} \int_{\R^2} \big|\widehat{\mathbbm{1}_{A}}(x,\xi)\big|^2 e^{-2\pi s^2 |\xi|^2} |\xi|^2 \,\textup{d}\xi \,\textup{d}x \,s\,\textup{d}s \,\textup{d}\alpha \,\frac{\textup{d}t}{t} .
\end{align*}
Interchanging the integrals in $\alpha$ and $s$, 
\[ \int_{\R} \int_{te^{-(m-1)\alpha-1}}^{te^{-(m-1)\alpha}} \cdots \,\textup{d}s \,\textup{d}\alpha
= \int_{0}^{\infty} \int_{-(\log(t/s)+1)/(m-1)}^{-(\log(t/s))/(m-1)} \cdots \,\textup{d}\alpha \,\textup{d}s, \]
and then performing the integration respectively in $\alpha$, $s$, $t$, and $\xi$,
we see that the expression from the right hand side of the previous display is actually equal to
\[ \frac{1}{4\pi(m-1)} \theta^{-4m+4} R^{-m-1} \Big(\log\frac{1}{\varepsilon}\Big)^2 \int_{\R^{m-1}} \big\| \mathbbm{1}_{A}(x,\cdot) \big\|_{\textup{L}^2(\R^2)}^2 \,\textup{d}x. \]
Note that the integral in the expression above is $|A|$, which is at most $R^{m+1}$.
\end{proof}

\begin{proof}[Proof of Theorem \ref{thm:simpldensity}]
For a given $\delta\in(0,1]$ we can choose 
\[ \varepsilon\sim\delta^{2(m+1)(2m-1)}, \]
so that Lemmas \ref{lm:simplstruct} and \ref{lm:simplunif} give
\[ \big|\mathcal{N}_{\lambda}^{0}(A;R)-\mathcal{N}_{\lambda}^{\varepsilon}(A;R)\big| \leq \frac{1}{3} \mathcal{N}_{\lambda}^{1}(A;R) \]
whenever $R\geq\lambda\geq1$.
Next, using Lemma \ref{lm:simplerr} we can find a positive number
\begin{align*} 
J & \sim \theta^{-4m+4} \delta^{-2(m+1)(2m-1)} \Big(\log\frac{1}{\varepsilon}\Big)^2 \\
& \sim \Big(\frac{1}{\delta}\Big)^{2(m+1)(4m-3)} \Big(1+\log\frac{1}{\delta}\Big)^2
\end{align*}
and then a number $\beta\in[0,J]$ such that
\[ \big|\mathcal{N}_{e^{\beta}}^{\varepsilon}(A;R)-\mathcal{N}_{e^{\beta}}^{1}(A;R)\big| 
\leq \Big( \frac{1}{J} \int_{0}^{J} \big( \mathcal{N}_{e^{\alpha}}^{\varepsilon}(A;R)-\mathcal{N}_{e^{\alpha}}^{1}(A;R) \big)^2 \,\textup{d}\alpha \Big)^{1/2}
\leq \frac{1}{3} \mathcal{N}_{e^\beta}^{1}(A;R), \]
as soon as $R$ is large enough that we can actually apply Lemma \ref{lm:simplstruct} for any such $\lambda=e^\beta$, i.e., 
\begin{equation}\label{eq:chooseRforJ}
R \geq e^J.
\end{equation}
From decomposition \eqref{eq:decomposition} we finally get
\[ \mathcal{N}_{e^{\beta}}^{0}(A;R) \geq \frac{1}{3}\mathcal{N}_{e^{\beta}}^{1}(A;R) > 0, \]
which guarantees existence in $A$ of the desired configuration \eqref{eq:wefoundsimplex} spanning a simplex of volume $1$.
Note that \eqref{eq:chooseRforJ} is in turn guaranteed if, say, 
\[ R\geq\exp\big(C\delta^{-9m^2}\big) \]
for a large constant $C$.
The last condition can be rewritten precisely as the assumption \eqref{eq:conddens} of Theorem \ref{thm:simpldensity}(a):
\[ \delta \geq \Big(\frac{C}{\log R}\Big)^{1/(9m^2)}. \qedhere \]
\end{proof}

\begin{proof}[Proof of Corollary \ref{cor:simpldensity}]
Take a set $A\subseteq\R^n$ of positive upper Banach density.
Let us first concentrate on the case $V=1$.
Denote
\[ R_0 := \exp\big(2C_m\bar{\delta}_{n}(A)^{-9m^2}\big), \]
where $C_m$ is the constant from Theorem \ref{thm:simpldensity}.
If we can find an axes-aligned cube $Q\subset\R^n$ of side-length $R_0$ on which $A$ has density greater than $(C_m/\log R_0)^{1/(9m^2)}$, then part (a) of that same theorem will apply to this cube $Q$, only translated from the position $[0,R_0]^n$.
It will give us a right $m$-simplex of volume $1$ with vertices in $A\cap Q$.
Moreover, the ratio of lengths of any two of its perpendicular edges will be at most 
\[ \frac{R_0^m}{m!} \leq \exp\big(2mC_m\bar{\delta}_{n}(A)^{-9m^2}\big) , \] 
as desired.

Thus, suppose that, on the contrary, for every such cube $Q$ of side-length $R_0$ we have
\begin{equation}\label{eq:auxcub}
\frac{|A\cap Q|}{|Q|} \leq \Big(\frac{C_m}{\log R_0}\Big)^{1/(9m^2)}.
\end{equation}
Take some $R\geq R_0$ and an arbitrary $x\in\R^n$. We can easily partition the cube $x+[0,R]^n$ into a collection $\mathcal{Q}$ of disjoint cubes of side-length $R_0$ and a remainder of volume less than $n R_0 R^{n-1}$.
By the assumption \eqref{eq:auxcub} applied to each cube $Q\in\mathcal{Q}$, we conclude that
\[ \frac{|A\cap(x+[0,R]^n)|}{R^n} \leq \Big(\frac{C_m}{\log R_0}\Big)^{1/(9m^2)} + \frac{n R_0}{R}. \]
Taking the supremum in $x$ and the upper limit in $R$, we obtain
\[ \bar{\delta}_{n}(A) \leq \Big(\frac{C_m}{\log R_0}\Big)^{1/(9m^2)}, \]
which contradicts the choice of $R_0$. 

Finally, for an arbitrary $V>0$ the set $V^{-1/m}A$ clearly has the same upper Banach density $\bar{\delta}_{n}(A)$, so the previous case applies to it and gives a right simplex of volume $1$ with vertices in $V^{-1/m}A$ and with controlled ratios of lengths of its perpendicular edges. It remains to scale the whole picture back by factor $V^{1/m}$ and thus obtain the desired simplex of volume $V$.
\end{proof}


\section{Large rectangular boxes: proof of Theorem \ref{thm:boxdensity}}
Fix $n\geq m+1$. If the statement in part (a) of the theorem did not hold, then there would exist a sequence of positive numbers $(\lambda_j)_{j=1}^{\infty}$ converging to $\infty$ such that for every index $j$ the set $A$ does not contain all vertices of any $m$-dimensional rectangular box (aligned as claimed in the theorem) in $\R^n$ with volume precisely $\lambda_j^m$.
Passing to a subsequence we could also ensure $\lambda_{j+1}\geq2\lambda_j$ for each index $j$.
The usual localization argument will allow us to work on a sufficiently large cube $[0,R]^n$ and with sufficiently (but finitely) many scales $\lambda_j$, $1\leq j\leq J$.
Indeed, leaving the choice of $J$ (depending on $\bar{\delta}_n(A)$) for later, we can find a number $R$ larger than $\lambda_J$ for which there exists $x\in\R^n$ such that
\[ \frac{|A\cap(x+[0,R]^n)|}{R^n} \geq \delta := \frac{\bar{\delta}_n(A)}{2}. \]
The set $A' := (A-x)\cap[0,R]^n$ is then a subset of the cube $[0,R]^n$ with density at least $\delta$ and it still does not contain the required boxes of volumes $\lambda_j^m$, $1\leq j\leq J$. 

Thus, it is sufficient to work, from the very beginning, with a measurable subset $A$ of a large cube $[0,R]^n$ with density $\delta>0$. We will obtain a contradiction with the assumption that the $m$-volumes of the (appropriately aligned) boxes with vertices in $A$ omit the numbers $\lambda_j^m$, $1\leq j\leq J$.
Here we work with a large number $J$ that we can choose depending on $\delta$, but the cube side-length $R>0$ needs to be an arbitrary number larger than some threshold (which in turn is allowed to depend on $\delta$ and $J$). 

It will be convenient to define
\begin{equation}\label{eq:boxchoiceoftheta}
\theta = m^{-1} 2^{-2^m n} \delta^{2^m}.
\end{equation}
Let $\sigma$ be the spherical measure on $\R^{n-m+1}$, supported on the standard unit sphere in that space and having the total mass equal to $1$.
The appropriate pattern-counting form is now, for every $\lambda\in(0,\infty)$,
\begin{align*}
\mathcal{N}_{\lambda}^{0}(A;R) :=\, 
& \int_{\R^{m-1}} \int_{\R^{n-m+1}} \int_{\R^{m-1}} \int_{\R^{n-m+1}} \\
& \Big(\prod_{(r_1,\ldots,r_m)\in\{0,1\}^m} \mathbbm{1}_A(x_1+r_1 u_1, \ldots, x_{m-1}+r_{m-1}u_{m-1}, y+r_m v)\Big) \\
& \textup{d}\sigma_{\lambda^m|u_1\cdots u_{m-1}|^{-1}}(v) \,\lambda^{-m+1} \Big(\prod_{k=1}^{m-1}\mathbbm{1}_{[-\lambda,-\theta\lambda]\cup[\theta\lambda,\lambda]}(u_k)\Big) \,\textup{d}u \,R^{-n} \,\textup{d}y \,\textup{d}x ,
\end{align*}
where it is understood that 
\[ u=(u_1,\ldots,u_{m-1}),\quad x=(x_1,\ldots,x_{m-1}). \]
The smoothed counting form is defined, for every $\varepsilon\in(0,1]$, as
\begin{align*}
\mathcal{N}_{\lambda}^{\varepsilon}(A;R) :=\, 
& \int_{\R^{m-1}} \int_{\R^{n-m+1}} \int_{\R^{m-1}} \int_{\R^{n-m+1}} \\
& \Big(\prod_{(r_1,\ldots,r_m)\in\{0,1\}^m} \mathbbm{1}_A(x_1+r_1 u_1, \ldots, x_{m-1}+r_{m-1}u_{m-1}, y+r_m v)\Big) & \\
& (\sigma \ast \mathbbm{g}_{\varepsilon})_{\lambda^m|u_1\cdots u_{m-1}|^{-1}}(v) \,\textup{d}v \,\lambda^{-m+1} \Big(\prod_{k=1}^{m-1}\mathbbm{1}_{[-\lambda,-\theta\lambda]\cup[\theta\lambda,\lambda]}(u_k)\Big) \,\textup{d}u \,R^{-n} \,\textup{d}y \,\textup{d}x ,
\end{align*}
The analogue of Lemma \ref{lm:simpconv} holds and is shown in exactly the same way.
We will apply the decomposition \eqref{eq:decomposition} here as well.
The following three lemmas respectively handle the structured, uniform, and error parts.

\begin{lemma}\label{lm:boxstruct}
For numbers $0<\lambda\leq R$ and a measurable set $A\subseteq[0,R]^n$ with density $\delta=|A|/R^n$ we have
\[ \mathcal{N}_{\lambda}^{1}(A;R) \gtrsim_\delta 1. \]
\end{lemma}

The lower bound that we prove is, in fact, a concrete one, involving a certain power of $\delta$, but Theorem \ref{thm:boxdensity} is not quantitative anyway.

\begin{proof}[Proof of Lemma \ref{lm:boxstruct}]
For arbitrary $u_1,\ldots,u_{m-1}\in[\theta\lambda,\lambda]$ we can estimate the convolution of the measure $\sigma$ with the Gaussian $\mathbbm{g}$ using \eqref{eq:convlowerbd} as
\begin{align*}
(\sigma\ast\mathbbm{g})_{\lambda^m|u_1\cdots u_{m-1}|^{-1}} & \gtrsim (\lambda^{-m}|u_1\cdots u_{m-1}|)^{n-m+1} \mathbbm{1}_{[-\lambda^m|u_1\cdots u_{m-1}|^{-1}, \lambda^m|u_1\cdots u_{m-1}|^{-1}]^{n-m+1}} \\
& \geq \theta^{(m-1)(n-m+1)} \lambda^{-n+m-1} \mathbbm{1}_{[-\lambda,\lambda]^{n-m+1}}.
\end{align*}
Similarly as in the previous section, we bound 
\begin{equation}\label{eq:boxlowerN1}
\mathcal{N}_{\lambda}^{1}(A;R) \gtrsim \theta^{(m-1)(n-m+1)} (R\lambda)^{-n} \int_{D\setminus E} \Phi(u,v,x,y) \,\textup{d}(u,v,x,y),
\end{equation}
where
\begin{align*}
D &:= [-\lambda,\lambda]^{m-1} \times [-\lambda,\lambda]^{n-m+1} \times [0,R]^{m-1} \times [0,R]^{n-m+1}, \\
E & := \big\{ (u,v,x,y)\in D : |u_k|<\theta\lambda\text{ for some }1\leq k\leq m-1 \big\}, \\
\Phi(u,v,x,y) & := \prod_{(r_1,\ldots,r_m)\in\{0,1\}^m} \mathbbm{1}_A(x_1+r_1 u_1, \ldots, x_{m-1}+r_{m-1}u_{m-1}, y+r_m v).
\end{align*}
Split $[0,\lceil R/\lambda\rceil\lambda]\supseteq[0,R]$ into intervals of length $\lambda$, denoting the resulting partition by $\mathcal{I}$, and split $[0,\lceil R/\lambda\rceil\lambda]^{n-m+1}\supseteq[0,R]^{n-m+1}$ into cubes of side-length $\lambda$, denoting the obtained partition by $\mathcal{Q}$.
Then
\begin{equation}\label{eq:boxcard}
\mathop{\textup{card}}\mathcal{I} \leq \frac{2R}{\lambda}, \quad \mathop{\textup{card}}\mathcal{Q} \leq \Big(\frac{2R}{\lambda}\Big)^{n-m+1}. 
\end{equation}

The integral of $\Phi$ over $D$ is at least
\begin{align*}
\sum_{\substack{I_1,\ldots,I_{m-1}\in\mathcal{I}\\Q\in\mathcal{Q}}} \int_{I_1} \int_{I_1} \cdots \int_{I_{m-1}} \int_{I_{m-1}} \int_{Q} \int_{Q} \prod_{(r_1,\ldots,r_m)\in\{0,1\}^m} \mathbbm{1}_A(x_1^{r_1}, \ldots, x_{m-1}^{r_{m-1}}, y^{r_m}) & \\
\textup{d}y^0 \,\textup{d}y^1 \,\textup{d}x_{m-1}^0 \,\textup{d}x_{m-1}^1 \cdots \textup{d}x_1^0 \,\textup{d}x_1^1 & ,
\end{align*}
(Note that here, and a few times later, zeros and ones in the exponents are simply upper indices.)
This is, by the Cauchy--Schwarz inequality for the Gowers-box inner products \cite{Shk06,GT08,Tao07} (or simply by repeated applications of the ordinary Cauchy--Schwarz inequality as in \cite[\S7.1]{DK22}), followed by Jensen's inequality for discrete sums and by \eqref{eq:boxcard}, greater than or equal to
\begin{align*}
& \sum_{\substack{I_1,\ldots,I_{m-1}\in\mathcal{I}\\Q\in\mathcal{Q}}} (|I_1|\cdots|I_{m-1}||Q|)^{2-2^m} \Big( \int_{I_1} \cdots \int_{I_{m-1}} \int_{Q} \mathbbm{1}_A(x_1, \ldots, x_{m-1}, y) 
\,\textup{d}y \,\textup{d}x_{m-1} \cdots \textup{d}x_1 \Big)^{2^m} \\
& = \sum_{\substack{I_1,\ldots,I_{m-1}\in\mathcal{I}\\Q\in\mathcal{Q}}} \lambda^{(2-2^m)n} |A\cap(I_1\times\cdots\times I_{m-1}\times Q)|^{2^m} \\
& \geq \big((\mathop{\textup{card}}\mathcal{I})^{m-1} (\mathop{\textup{card}}\mathcal{Q})\big)^{1-2^m} \lambda^{(2-2^m)n} |A|^{2^m}
\geq 2^{(1-2^m)n} (R\lambda)^n \delta^{2^m}.
\end{align*}

The integral of $\Phi$ over $E$ is at most
\[ |E| \leq (m-1) \theta (R\lambda)^n. \]
Thus, by \eqref{eq:boxlowerN1},
\[ \mathcal{N}_{\lambda}^{1}(A;R) \gtrsim \theta^{(m-1)(n-m+1)} \big( 2^{(1-2^m)n} \delta^{2^m}- m\theta \big). \]
The choice \eqref{eq:boxchoiceoftheta} for $\theta$ is a good one and it gives
\[ \mathcal{N}_{\lambda}^{1}(A;R) \gtrsim \delta^{2^m((m-1)(n-m+1)+1)} \gtrsim_\delta 1. \qedhere \]
\end{proof}

\begin{lemma}\label{lm:boxunif}
For every $\lambda,R\in(0,\infty)$, every $\varepsilon\in(0,1]$, and every measurable set $A\subseteq[0,R]^n$ we have
\[ \big|\mathcal{N}_{\lambda}^{0}(A;R)-\mathcal{N}_{\lambda}^{\varepsilon}(A;R)\big| \lesssim \varepsilon^{1/2}. \]
\end{lemma}

\begin{proof}[Proof of Lemma \ref{lm:boxunif}]
The proof is almost the same as that of Lemma \ref{lm:simplunif}.
Here we find convenient to introduce the function
\[ f(x,y;u) := \prod_{(r_1,\ldots,r_{m-1})\in\{0,1\}^{m-1}} \mathbbm{1}_A(x_1+r_1 u_1, \ldots, x_{m-1}+r_{m-1}u_{m-1}, y) \]
and the number 
\[ a = a(\lambda,u_1,\ldots,u_{m-1}) \]
defined as
\[ a := \lambda^m|u_1\cdots u_{m-1}|^{-1}. \]
For any $0<\tau<\varepsilon$ integrate the heat equation \eqref{eq:heatequ} as in \eqref{eq:heatequ2} to obtain 
\begin{align*}
\mathcal{N}_{\lambda}^{\tau}(A;R)-\mathcal{N}_{\lambda}^{\varepsilon}(A;R) = 
& - \int_{\tau}^{\varepsilon} \int_{\R^{m-1}} \int_{\R^{m-1}} \int_{\R^{n-m+1}} \int_{\R^{n-m+1}} f(x,y;u) f(x,y+v;u) \\
& \times R^{-n} \lambda^{-m+1} (\sigma \ast \mathbbm{k}_{t})_{a}(v)
\Big(\prod_{k=1}^{m-1}\mathbbm{1}_{[-\lambda,-\theta\lambda]\cup[\theta\lambda,\lambda]}(u_k)\Big)
\,\textup{d}v \,\textup{d}y \,\textup{d}u \,\textup{d}x \,\frac{\textup{d}t}{2\pi t} .
\end{align*}
Write
\[ \int_{\R^{n-m+1}} \int_{\R^{n-m+1}} f(x,y;u) f(x,y+v;u) (\sigma \ast \mathbbm{k}_{t})_{a}(v) \,\textup{d}v \,\textup{d}y \]
as
\[ \int_{\R^{n-m+1}} \big|\widehat{f}(x,\xi;u)\big|^2 \,\widehat{\sigma}(a\xi) \,\widehat{\mathbbm{k}}(ta\xi) \,\textup{d}\xi, \]
the absolute value of which is bounded, thanks to \eqref{eq:measurewithk}, by a constant times
\[ t^{1/2} \int_{\R^{n-m+1}} \big|\widehat{f}(x,\xi;u)\big|^2 \,\textup{d}\xi
= t^{1/2} \int_{\R^{n-m+1}} f(x,y;u)^2 \,\textup{d}y 
\leq t^{1/2} R^{n-m+1}. \]
Thus,
\begin{align*}
\big|\mathcal{N}_{\lambda}^{\tau}(A;R)-\mathcal{N}_{\lambda}^{\varepsilon}(A;R)\big|
\lesssim \int_{\tau}^{\varepsilon} \int_{[0,R]^{m-1}} \int_{[-\lambda,\lambda]^{m-1}}
t^{1/2} R^{-m+1} \lambda^{-m+1} \,\textup{d}u \,\textup{d}x \,\frac{\textup{d}t}{t} & \\
\lesssim \int_{\tau}^{\varepsilon} t^{-1/2}\,\textup{d}t \lesssim \varepsilon^{1/2} & .
\end{align*}
Letting $\tau\to0$ we prove the lemma.
\end{proof}

\begin{lemma}\label{lm:boxerr}
For every $R\in(0,\infty)$, every $\varepsilon\in(0,1]$, every choice of positive numbers $(\lambda_j)_{1\leq j\leq J}$ satisfying $\lambda_{j+1}\geq2\lambda_j$, and every measurable set $A\subseteq[0,R]^n$ we have
\[ \sum_{j=1}^{J} \big|\mathcal{N}_{\lambda_j}^{\varepsilon}(A;R)-\mathcal{N}_{\lambda_j}^{1}(A;R)\big| \lesssim_{\delta,\varepsilon} 1. \]
\end{lemma}

\begin{proof}[Proof of Lemma \ref{lm:boxerr}]
Let $f$ and $a$ be defined as in the proof of Lemma \ref{lm:boxunif}. In the same way we arrive at
\begin{align*}
& \mathcal{N}_{\lambda}^{\varepsilon}(A;R)-\mathcal{N}_{\lambda}^{1}(A;R) \\
& = - \int_{\varepsilon}^{1} \int_{t\lambda/(10e)}^{t\lambda/10} \int_{\R^{m-1}} \int_{\R^{m-1}} \int_{\R^{n-m+1}} \int_{\R^{n-m+1}} f(x,y;u) f(x,y+v;u) \\
& \qquad \times R^{-n}\lambda^{-m+1} (\sigma \ast \mathbbm{k}_{t})_{a}(v)
\Big(\prod_{k=1}^{m-1}\mathbbm{1}_{[-\lambda,-\theta\lambda]\cup[\theta\lambda,\lambda]}(u_k)\Big)
\,\textup{d}v \,\textup{d}y \,\textup{d}u \,\textup{d}x \,\frac{\textup{d}s}{s} \,\frac{\textup{d}t}{2\pi t} ,
\end{align*}
noting that here we have also introduced the integral of $1/s$ in the variable 
\begin{equation*}
s\in[10^{-1}e^{-1}t\lambda,10^{-1}t\lambda],
\end{equation*}
which equals $1$, so it does not change the value of the whole expression.
The shorthand notation $r$ has the same meaning as in \eqref{eq:defofr1} and \eqref{eq:defofr2}, except that $s$ is now taken from a different interval.
Once again, convolution identities \eqref{eq:convidkg} and \eqref{eq:convidhh} give
\[ \mathbbm{k}_{ta} = \Big(\frac{ta}{s}\Big)^2 \sum_{l=1}^{n-m+1} \mathbbm{h}^{(l)}_{s} \ast \mathbbm{h}^{(l)}_{s} \ast \mathbbm{g}_{r}, \]
so
\[ - \int_{\R^{n-m+1}} \int_{\R^{n-m+1}} f(x,y;u) f(x,y+v;u) (\sigma \ast \mathbbm{k}_{t})_{a}(v) \,\textup{d}v \,\textup{d}y \]
can be rewritten as
\begin{align*} 
\Big(\frac{ta}{s}\Big)^2 \sum_{l=1}^{n-m+1} \int_{\R^{n-m+1}} \int_{\R^{n-m+1}} & \Big( \int_{\R^{n-m+1}} f(x,y;u) \mathbbm{h}^{(l)}_{s}(y-z) \,\textup{d}y \Big) \\
\times & \Big( \int_{\R^{n-m+1}} f(x,y';u) \mathbbm{h}^{(l)}_{s}(y'-z') \,\textup{d}y' \Big) (\sigma_{a}\ast\mathbbm{g}_{r})(z-z') \,\textup{d}z \,\textup{d}z'.
\end{align*}
Under the current assumption on the numbers $u_k$ we have $\lambda \leq a \leq \theta^{-m+1} \lambda$, which implies $s \sim t\lambda \sim_\delta t a$ and, taking absolute values, we bound the difference of the counting forms as
\begin{align*}
\big| \mathcal{N}_{\lambda}^{\varepsilon}(A;R) & - \mathcal{N}_{\lambda}^{1}(A;R) \big|
\lesssim_\delta  R^{-n} \sum_{l=1}^{n-m+1} \int_{\varepsilon}^{1} \int_{t\lambda/(10e)}^{t\lambda/10} \int_{\R^{m-1}} \int_{\R^{m-1}} \int_{\R^{n-m+1}} \int_{\R^{n-m+1}} \\
& \Big| \int_{\R^{n-m+1}} f(x,y;u) \mathbbm{h}^{(l)}_{s}(y-z) \,\textup{d}y \Big| 
\Big| \int_{\R^{n-m+1}} f(x,y';u) \mathbbm{h}^{(l)}_{s}(y'-z') \,\textup{d}y' \Big| \\
& \times (\sigma_a \ast \mathbbm{g}_{r})(z-z') \,\lambda^{-m+1} \Big(\prod_{k=1}^{m-1}\mathbbm{1}_{[-\lambda,-\theta\lambda]\cup[\theta\lambda,\lambda]}(u_k)\Big) 
\,\textup{d}z \,\textup{d}z' \,\textup{d}u \,\textup{d}x \,\frac{\textup{d}s}{s} \,\frac{\textup{d}t}{t} .
\end{align*}
By the Cauchy--Schwarz inequality in $t,s,x,u,z,z'$ and the symmetry between $z$ and $z'$, we estimate this by
\begin{align*}
& R^{-n} \sum_{l=1}^{n-m+1} \int_{\varepsilon}^{1} \int_{t\lambda/(10e)}^{t\lambda/10} \int_{\R^{m-1}} \int_{\R^{m-1}} \int_{\R^{n-m+1}} \\
& \quad \Big( \int_{\R^{n-m+1}} f(x,y;u) \mathbbm{h}^{(l)}_{s}(y-z) \,\textup{d}y \Big)^2 \,\lambda^{-m+1} \Big(\prod_{k=1}^{m-1}\mathbbm{1}_{[-\lambda,-\theta\lambda]\cup[\theta\lambda,\lambda]}(u_k)\Big) 
\,\textup{d}z \,\textup{d}u \,\textup{d}x \,\frac{\textup{d}s}{s} \,\frac{\textup{d}t}{t} ,
\end{align*}
where we have already eliminated $(\sigma_a \ast \mathbbm{g}_{r})(z-z')$ by evaluating the integral in $z'$.
Dominate $\mathbbm{1}_{[-1,1]}$ by the Gaussian $\mathbbm{g}_{s/\lambda}$ using \eqref{eq:lowerbd2} and recalling $s\sim_\varepsilon \lambda$; then dilate by $\lambda$ to obtain
\[ \lambda^{-1} \mathbbm{1}_{[-\lambda,-\theta\lambda]\cup[\theta\lambda,\lambda]}(u_k) \lesssim_\varepsilon \mathbbm{g}_{s}(u_k). \]
Then we expand out the square $(\int_{\R^{n-m+1}}\cdots\textup{d}y)^2$ and collapse back the convolution using \eqref{eq:convidhh}: 
\[ \sum_{l=1}^{n-m+1} \mathbbm{h}^{(l)}_{s} \ast \mathbbm{h}^{(l)}_{s} = \frac{1}{2} \mathbbm{k}_{s\sqrt{2}}. \]
That way we finally end up with
\begin{align*}
\big|\mathcal{N}_{\lambda}^{\varepsilon}(A;R)-\mathcal{N}_{\lambda}^{1}(A;R)\big|
& \lesssim_{\delta,\varepsilon} - R^{-n} \int_{\varepsilon\lambda/(10e)}^{\lambda/10} \int_{\R^{m-1}} \int_{\R^{m-1}} \int_{\R^{n-m+1}} \int_{\R^{n-m+1}} \\
& \quad f(x,y^0;u) f(x,y^1;u) \mathbbm{k}_{s\sqrt{2}}(y^0-y^1) \,\Big(\prod_{k=1}^{m-1}\mathbbm{g}_{s}(u_k)\Big) 
\,\textup{d}y^0 \,\textup{d}y^1 \,\textup{d}u \,\textup{d}x \,\frac{\textup{d}s}{s} .
\end{align*}
Due to our assumption on the exponential growth of the numbers $\lambda_j$, every positive number $s$ lies in a bounded number of intervals $[\varepsilon\lambda_j/(10e),\lambda_j/10]$, $1\leq j\leq J$, with a bound depending on $\varepsilon$ alone. Therefore, applying the last estimate with $\lambda=\lambda_j$ and summing in $j$, we get
\[ \sum_{j=1}^{J} \big|\mathcal{N}_{\lambda_j}^{\varepsilon}(A;R)-\mathcal{N}_{\lambda_j}^{1}(A;R)\big| \lesssim_{\delta,\varepsilon} R^{-n} \,\Theta^{(m)}_{1,\ldots,1,\sqrt{2}}(\mathbbm{1}_A). \]
Here we define, more generally,
\begin{align*}
\Theta^{(k)}_{\gamma_1,\ldots,\gamma_m}(F) :=
- \int_{0}^{\infty} \int_{\mathbb{R}^{d_1}\times\mathbb{R}^{d_1}\times\cdots\times\mathbb{R}^{d_m}\times\mathbb{R}^{d_m}} \Big( \prod_{(r_1,\ldots,r_n)\in\{0,1\}^m} F(x_1^{r_1}, \ldots, x_m^{r_m}) \Big)
\, \mathbbm{k}_{s\gamma_k}(x_k^0-x_k^1) & \\ 
\times \Big(\prod_{\substack{1\leq i\leq m\\ i\neq k}}\mathbbm{g}_{s\gamma_i}(x_i^0-x_i^1)\Big)
\,\textup{d}(x_1^0,x_1^1,x_2^0,x_2^1,\ldots,x_m^0,x_m^1) \,\frac{\textup{d}s}{s} & 
\end{align*}
for any $1\leq k\leq m$, any parameters $\gamma_1,\ldots,\gamma_m\in(0,\infty)$, and a real-valued bounded and compactly supported measurable function $F$ on $\R^{d_1+d_2+\cdots+d_m}$.
In our case, $d_1=\cdots=d_{m-1}=1$ and $d_m=n-m+1$,
but the greater generality will only make the arguments more elegant and symmetric.

The proof of Lemma \ref{lm:boxerr} will be complete once we know the estimate
\begin{equation}\label{eq:SBLestimate}
\Theta^{(k)}_{\gamma_1,\ldots,\gamma_m}(F) \leq 2\pi \|F\|_{\textup{L}^{2^m}(\mathbb{R}^d)}^{2^m} 
\end{equation}
and apply it to the function $F=\mathbbm{1}_A$, for which
\[ \|F\|_{\textup{L}^{2^m}(\mathbb{R}^n)}^{2^m} = |A| \leq R^n. \]
Indeed, $\Theta^{(k)}_{\gamma_1,\ldots,\gamma_m}$ is an example of an \emph{entangled singular integral form} \cite{Kov11,Dur15,Dur17} and inequality \eqref{eq:SBLestimate} is, up to an unimportant constant, a particular case of the recently studied \emph{singular Brascamp--Lieb estimates} \cite{DT20,DT21survey,DST22}. 
However, in just a few more lines, we can give a self-contained simple proof of this very special case. It has already appeared in \cite[\S7.2]{DK22} (based on the earlier ideas from \cite{Kov12,Dur15}), but we include it here for completeness.

First, observe that 
\[ \Theta^{(k)}_{\gamma_1,\ldots,\gamma_m}(F) \geq 0 \]
for each $1\leq k\leq m$ and every real-valued function $F$. 
This is seen by denoting
\begin{align*} 
f_k(x) = & \, f_k(x; x_1^0, x_1^1, \ldots, x_{k-1}^0, x_{k-1}^1, x_{k+1}^0, x_{k+1}^1, \ldots, x_m^0, x_m^1) \\
:= & \prod_{(r_1,\ldots,r_{k-1},r_{k+1},\ldots,r_n)\in\{0,1\}^{m-1}} F(x_1^{r_1}, \ldots, x_{k-1}^{r_{k-1}}, x, x_{k+1}^{r_{k+1}}, \ldots, x_m^{r_m}) 
\end{align*}
and using \eqref{eq:convidhh} again, this time to rewrite
\begin{align*}
& - \int_{\mathbb{R}^{d_k}} \int_{\mathbb{R}^{d_k}} f_k(x_k^0) f_k(x_k^1) \,\mathbbm{k}_{s\gamma_k}(x_k^0-x_k^1) \,\textup{d}x_k^0 \,\textup{d}x_k^1 \\
& = 2 \sum_{l=1}^{d_k} \int_{\mathbb{R}^{d_k}} \Big( \int_{\mathbb{R}^{d_k}} f_k(x) \,\mathbbm{h}^{(l)}_{s\gamma_k/\sqrt{2}}(x-z) \,\textup{d}x \Big)^2 \,\textup{d}z \geq 0.
\end{align*}
Essentially, we were merely undoing the last step of the computation that lead us to the particular case $\Theta^{(m)}_{1,\ldots,1,\sqrt{2}}$ in the first place.
Next, using the product rule for differentiation and the heat equation \eqref{eq:heatequ}, we obtain
\begin{equation}\label{eq:ooproductrule}
\frac{\partial}{\partial s} \prod_{i=1}^{m} \mathbbm{g}_{s\gamma_i}(u_i) = \sum_{k=1}^{m} \frac{1}{2\pi s} \mathbbm{k}_{s\gamma_k}(u_k)  
\prod_{\substack{1\leq i\leq m\\ i\neq k}}\mathbbm{g}_{s\gamma_i}(u_i),
\end{equation}
from which, integrating over $s\in(0,\infty)$, we easily get
\[ \sum_{k=1}^{m} \Theta^{(k)}_{\gamma_1,\ldots,\gamma_m}(F) = 2\pi \|F\|_{\textup{L}^{2^m}(\mathbb{R}^d)}^{2^m}. \]
Since the above summands are nonnegative, each of them is dominated by the whole sum, so \eqref{eq:SBLestimate} follows.
\end{proof}

\begin{proof}[Completion of the proof of Theorem \ref{thm:boxdensity}]
The standing assumption is
\[ \mathcal{N}_{\lambda_j}^{0}(A;R) = 0 \]
for every index $1\leq j\leq J$.
By Lemmas \ref{lm:boxstruct}--\ref{lm:boxerr} we can first choose a sufficiently small $\varepsilon>0$ (depending on $\delta$) such that
\[ \big|\mathcal{N}_{\lambda}^{0}(A;R)-\mathcal{N}_{\lambda}^{\varepsilon}(A;R)\big| \leq \frac{1}{3}\mathcal{N}_{\lambda}^{1}(A;R) \]
for every $0<\lambda\leq R$, and then a sufficiently large $J$ (depending on $\delta$ and $\varepsilon$) that the pigeonhole principle gives an index $1\leq j\leq J$ such that
\[ \big|\mathcal{N}_{\lambda_j}^{\varepsilon}(A;R)-\mathcal{N}_{\lambda_j}^{1}(A;R)\big| \leq \frac{1}{3}\mathcal{N}_{\lambda_j}^{1}(A;R). \]
Note that we were able to apply Lemma \ref{lm:boxstruct} as soon as $R\geq\lambda_J$.
Then decomposition \eqref{eq:decomposition} gives
\[ 0 = \mathcal{N}_{\lambda_j}^{0}(A;R) \geq \frac{1}{3}\mathcal{N}_{\lambda_j}^{1}(A;R) > 0 \]
for that same index $j$, which is a contradiction.
\end{proof}


\section{Hyperbolic embeddings: proof of Theorem \ref{thm:spacetimedensity}}
\label{sec:prooftrees}
Rotating the plane by $45^\circ$ and stretching it by factor $\sqrt{2}$ we can, more conveniently, replace the hyperbola $\mathscr{H}$ with the rotated hyperbola
\[ \mathscr{H}':\qquad x y = 1. \]
This will only affect the constant $C$ from the theorem statement.
Let $\mu$ be a Borel measure on $\R^2$ supported on $\mathscr{H}'$, defined via
\[ \int_{\R^2} f(x,y) \,\textup{d}\mu(x,y) = \int_{0}^{\infty} f\Big(x,\frac{1}{x}\Big) \,\varphi(x)\,\textup{d}x , \]
where $\varphi\colon(0,\infty)\to[0,\infty)$ is a compactly supported $\textup{C}^\infty$ function.
We also adjust $\varphi$ in a way that $\mu$ has total mass $1$.
Constants in the proof (appearing implicitly in the notation $\lesssim$ and $\gtrsim$) are allowed to depend on the exact choice of $\varphi$, which will not be notationally emphasized. 
They will also depend on the number $n$ and on the edge-lengths $a_1,\ldots,a_n\in(0,\infty)$, which are all regarded as fixed.

The pattern-counting form is, for $\lambda\in(0,\infty)$,
\[ \mathcal{N}_{\lambda}^{0}(A;R) := R^{-2} \int_{\R^2} \int_{(\R^2)^{n}} \mathbbm{1}_A(z) \Big(\prod_{k=1}^{n}\mathbbm{1}_A(z+u_k)\Big) 
\Big(\prod_{k=1}^{n} \textup{d}\mu_{a_k\lambda,a_k\lambda^{-1}}(u_k)\Big) \,\textup{d}z \]
and its smoothed counterpart is, after also introducing $\varepsilon\in(0,1]$,
\[ \mathcal{N}_{\lambda}^{\varepsilon}(A;R) := R^{-2} \int_{\R^2} \int_{(\R^2)^{n}} \mathbbm{1}_A(z) \Big(\prod_{k=1}^{n}\mathbbm{1}_A(z+u_k)\Big) 
\Big(\prod_{k=1}^{n} (\mu\ast \mathbbm{g}_{\varepsilon})_{a_k\lambda,a_k\lambda^{-1}}(u_k)\Big) \,\textup{d}u \,\textup{d}z. \]
Here $u$ is a shorter notation for the $n$-tuple $(u_1,\ldots,u_n)$.
As expected, $\mathcal{N}_{\lambda}^{0}(A;R)$ can again be obtained in the limit of $\mathcal{N}_{\lambda}^{\varepsilon}(A;R)$ as $\varepsilon\to0$, which is shown easily, as in Lemma \ref{lm:simpconv}.
For every $a,\lambda>0$ the support of the anisotropically dilated measure $\mu_{a\lambda,a\lambda^{-1}}$ lies on the hyperbola $\mathscr{H}'$ stretched by factor $a$.
Thus, $\mathcal{N}_{\lambda}^{0}(A;R)>0$ for any value of $\lambda>0$ clearly means that there exists the desired configuration \eqref{eq:hyperpoints} inside the set $A$.

Decomposition \eqref{eq:decomposition} will be applied one more time, so we proceed with bounding its three parts.

\begin{lemma}\label{lm:spacetimestruct}
For every $\lambda,R\in(0,\infty)$ such that $R^{-1}\leq\lambda\leq R$, and a measurable set $A\subseteq[0,R]^2$ of density $\delta=|A|/R^2$ we have
\[ \mathcal{N}_{\lambda}^{1}(A;R) \gtrsim \delta^{n+1}. \]
\end{lemma}

\begin{proof}[Proof of Lemma \ref{lm:spacetimestruct}]
We use a slight modification of \eqref{eq:convlowerbd} and scale it anisotropically to obtain
\[ (\mu_{a_k}\ast\mathbbm{g}_{a_k})_{\lambda,\lambda^{-1}} \gtrsim \mathbbm{1}_{[-\lambda,\lambda]\times[-\lambda^{-1},\lambda^{-1}]}. \]
Let $\mathcal{Q}$ be the partition of 
\[ [0,\lceil R/\lambda\rceil\lambda]\times[0,\lceil R\lambda\rceil\lambda^{-1}] \supseteq [0,R]^2 \] 
into $\lceil R/\lambda\rceil \lceil R\lambda\rceil\leq4R^2$ rectangles of size $\lambda\times\lambda^{-1}$.
Then
\begin{align*}
\mathcal{N}_{\lambda}^{1}(A;R) & \gtrsim R^{-2} \sum_{Q\in\mathcal{Q}} \int_{Q^{n+1}} \Big(\prod_{k=0}^{n}\mathbbm{1}_A(z_k)\Big) \,\textup{d}(z_k)_{0\leq k\leq n} 
= R^{-2} \sum_{Q\in\mathcal{Q}} |A\cap Q|^{n+1} \\
& \geq R^{-2} (\mathop{\textup{card}}\mathcal{Q})^{-n} \Big(\sum_{Q\in\mathcal{Q}} |A\cap Q| \Big)^{n+1}
\geq R^{-2} (4R^2)^{-n} |A|^{n+1} \sim \delta^{n+1}. \qedhere
\end{align*}
\end{proof}

\begin{lemma}\label{lm:spacetimeunif}
For every $\lambda,R\in(0,\infty)$, $\varepsilon\in(0,1]$, and a measurable set $A\subseteq[0,R]^2$ of density $\delta$ we have
\[ \big|\mathcal{N}_{\lambda}^{0}(A;R)-\mathcal{N}_{\lambda}^{\varepsilon}(A;R)\big| \lesssim \delta \varepsilon^{1/2}. \]
\end{lemma}

\begin{proof}[Proof of Lemma \ref{lm:spacetimeunif}]
Using the fundamental theorem of calculus, the heat equation \eqref{eq:heatequ}, and the product rule for differentiation, similarly as in \eqref{eq:ooproductrule}, we can write, for $0<\tau<\varepsilon$,
\begin{align*}
\mathcal{N}_{\lambda}^{\tau}(A;R)-\mathcal{N}_{\lambda}^{\varepsilon}(A;R) = 
& - R^{-2} \sum_{j=1}^{n} \int_{\tau}^{\varepsilon} \int_{\R^2} \int_{(\R^2)^{n}}  
\mathbbm{1}_A(z) \Big(\prod_{k=1}^{n}\mathbbm{1}_A(z+u_k)\Big) \\
& \times (\mu\ast \mathbbm{k}_{t})_{a_j\lambda,a_j\lambda^{-1}}(u_j) 
\Big(\prod_{\substack{1\leq k\leq n\\k\neq j}} (\mu\ast \mathbbm{g}_{t})_{a_k\lambda,a_k\lambda^{-1}}(u_k)\Big) 
\,\textup{d}u \,\textup{d}z \,\frac{\textup{d}t}{2\pi t} .
\end{align*}
A standard computation using Plancherel's theorem rewrites
\[ \int_{\R^2} \int_{\R^2} 
\mathbbm{1}_A(z) \mathbbm{1}_A(z+u_j)
(\mu\ast \mathbbm{k}_{t})_{a_j\lambda,a_j\lambda^{-1}}(u_j)
\,\textup{d}u_j \,\textup{d}z \]
as
\[ \int_{\R^2} \big|\widehat{\mathbbm{1}_A}(\zeta,\eta)\big|^2 \,\widehat{\mu}(a_j\lambda\zeta,a_j\lambda^{-1}\eta)
\,\widehat{\mathbbm{k}}(t a_j \lambda\zeta, t a_j \lambda^{-1}\eta)
\,\textup{d}(\zeta,\eta). \]
By \eqref{eq:measurewithk}, the absolute value of the last integral is (up to a constant) at most
\[ t^{1/2} \big\|\widehat{\mathbbm{1}_A}\big\|_{\textup{L}^2(\R^2)}^2  
= t^{1/2} \|\mathbbm{1}_A\|_{\textup{L}^2(\R^2)}^2 = t^{1/2} |A|. \]
That way we end up with
\begin{align*}
& \big|\mathcal{N}_{\lambda}^{\tau}(A;R)-\mathcal{N}_{\lambda}^{\varepsilon}(A;R)\big| \\
& \lesssim R^{-2} |A| \sum_{j=1}^{n} \int_{\tau}^{\varepsilon} \int_{(\R^2)^{n-1}}  
\Big(\prod_{\substack{1\leq k\leq n\\k\neq j}} (\mu\ast \mathbbm{g}_{t})_{a_k\lambda,a_k\lambda^{-1}}(u_k)\Big) 
\,\textup{d}(u_k)_{k\neq j} \,\frac{\textup{d}t}{t^{1/2}} \\
& \lesssim R^{-2} |A| \varepsilon^{1/2}
= \delta \varepsilon^{1/2},
\end{align*}
where we first integrated freely in the variables $u_k$ and then simply evaluated the integral in $t$.
It remains to let $\tau\to0$.
\end{proof}

Before we finally turn to the error part, let us prove a very concrete estimate for the Fourier transform $\widehat{\mathbbm{k}}$ of the Laplacian of the two-dimensional Gaussian. Despite being easy, in this property the hyperbolic scaling from the problem formulation is finally used.

\begin{lemma}\label{lm:kscalesnicely}
The estimate
\[ \int_{0}^{\infty} \Big|\widehat{\mathbbm{k}}\Big(s\zeta,\frac{\eta}{s}\Big)\Big| \,\frac{\textup{d}s}{s} \lesssim 1 \]
holds for every $(\zeta,\eta)\in\R^2$.
\end{lemma}

\begin{proof}[Proof of Lemma \ref{lm:kscalesnicely}]
From the last formula in \eqref{eq:Ftransforms} we see that the given integral equals
\[ \mathscr{K}(\zeta,\eta) := 4\pi^2 \int_{0}^{\infty} (s^2\zeta^2+s^{-2}\eta^2) e^{-\pi(s^2\zeta^2+s^{-2}\eta^2)} \,\frac{\textup{d}s}{s} . \]
In the case $\zeta\eta\neq0$, substitution $t=s\sqrt{|\zeta|/|\eta|}$ gives
\begin{align*}
\mathscr{K}(\zeta,\eta) & = 4\pi^2 |\zeta\eta| \int_{0}^{\infty} (t^2+t^{-2}) e^{-\pi|\zeta\eta|(t^2+t^{-2})} \,\frac{\textup{d}t}{t} \\
& \leq 8\pi^2 |\zeta\eta| \Big( \int_{1}^{\infty} t^2 e^{-\pi|\zeta\eta|t^2} \,\frac{\textup{d}t}{t} + \int_{0}^{1} t^{-2} e^{-\pi|\zeta\eta|t^{-2}} \,\frac{\textup{d}t}{t} \Big) \\
& \leq 16\pi^2 |\zeta\eta| \int_{0}^{\infty} t e^{-\pi|\zeta\eta|t^2} \,\textup{d}t = 8\pi.
\end{align*}
An even easier computation yields
$\mathscr{K}(\zeta,0)=2\pi$ for $\zeta\neq0$,
$\mathscr{K}(0,\eta)=2\pi$ for $\eta\neq0$,
and $\mathscr{K}(0,0)=0$.
\end{proof}

We will appreciate Lemma \ref{lm:kscalesnicely} more after the discussion in Section \ref{sec:remarksanalysts} on complications that the hyperbolic setting brings to some related analytical problems.
Now we can control the remaining term from the decomposition \eqref{eq:decomposition}.

\begin{lemma}\label{lm:spacetimeerr}
For every $R\in(0,\infty)$, $\varepsilon\in(0,1]$, and a measurable set $A\subseteq[0,R]^2$ of density $\delta$ we have
\[ \int_{\R} \big(\mathcal{N}_{e^\alpha}^{\varepsilon}(A;R)-\mathcal{N}_{e^\alpha}^{1}(A;R)\big)^2 \,\textup{d}\alpha \lesssim \delta^2 \Big(\log\frac{1}{\varepsilon}\Big)^2. \]
\end{lemma}

\begin{proof}[Proof of Lemma \ref{lm:spacetimeerr}]
Using the computation from the previous proof, we rewrite $\mathcal{N}_{\lambda}^{\varepsilon}(A;R)-\mathcal{N}_{\lambda}^{1}(A;R)$ as
\begin{align*}
- R^{-2} \sum_{j=1}^{n} \int_{\varepsilon}^{1} \int_{\R^2} \int_{(\R^2)^{n}}  
& \mathbbm{1}_A(z) \Big(\prod_{k=1}^{n}\mathbbm{1}_A(z+u_k)\Big) \\
& \times (\mu\ast \mathbbm{k}_{t})_{a_j\lambda,a_j\lambda^{-1}}(u_j) 
\Big(\prod_{\substack{1\leq k\leq n\\k\neq j}} (\mu\ast \mathbbm{g}_{t})_{a_k\lambda,a_k\lambda^{-1}}(u_k)\Big) 
\,\textup{d}u \,\textup{d}z \,\frac{\textup{d}t}{2\pi t} .
\end{align*}
For each fixed $1\leq j\leq n$ we introduce the variable $z'=z+u_j$, expand applying \eqref{eq:convidab} and \eqref{eq:convidaa},
\begin{align*}
- (\mu\ast \mathbbm{k}_{t})_{a_j\lambda,a_j\lambda^{-1}}(u_j) & = - (\mu\ast \mathbbm{k}_{t})_{a_j\lambda,a_j\lambda^{-1}}(z'-z) \\
& = 4 \sum_{l=1}^{2} \int_{\R^2} \mathbbm{h}^{(l)}_{a_jt\lambda/2,a_jt\lambda^{-1}/2}(z-w) 
\mathbbm{h}^{(l)}_{a_jt\lambda/2,a_jt\lambda^{-1}/2}(z'-w') \\
& \qquad\qquad\quad \times (\mu\ast \mathbbm{g}_{t/\sqrt{2}})_{a_j\lambda,a_j\lambda^{-1}}(w-w') \,\textup{d}w \,\textup{d}w' ,
\end{align*}
and use the Cauchy--Schwarz inequality to obtain
\begin{equation}\label{eq:justCS}
\big(\mathcal{N}_{\lambda}^{\varepsilon}(A;R)-\mathcal{N}_{\lambda}^{1}(A;R)\big)^2 \lesssim R^{-4} \mathcal{A}\, \mathcal{B}, 
\end{equation}
where
\begin{align*}
\mathcal{A} := \sum_{j=1}^{n} \sum_{l=1}^{2} & \int_{\varepsilon}^{1} \int_{(\R^2)^{n-1}} \int_{\R^2} \int_{\R^2} 
\bigg( \int_{\R^2} \mathbbm{1}_A(z) \Big(\prod_{\substack{1\leq k\leq n\\k\neq j}}\mathbbm{1}_A(z+u_k)\Big) 
\mathbbm{h}^{(l)}_{a_jt\lambda/2,a_jt\lambda^{-1}/2}(z-w) \,\textup{d}z \bigg)^2 \\
& \times (\mu\ast \mathbbm{g}_{t/\sqrt{2}})_{a_j\lambda,a_j\lambda^{-1}}(w-w') \Big(\prod_{\substack{1\leq k\leq n\\k\neq j}}(\mu\ast \mathbbm{g}_{t})_{a_k\lambda,a_k\lambda^{-1}}(u_k)\Big) 
\,\textup{d}w \,\textup{d}w' \,\textup{d}(u_k)_{k\neq j} \,\frac{\textup{d}t}{t} 
\end{align*}
and
\begin{align*}
\mathcal{B} := \sum_{j=1}^{n} \sum_{l=1}^{2} & \int_{\varepsilon}^{1} \int_{(\R^2)^{n-1}} \int_{\R^2} \int_{\R^2} 
\bigg( \int_{\R^2} \mathbbm{1}_A(z') 
\,\mathbbm{h}^{(l)}_{a_jt\lambda/2,a_jt\lambda^{-1}/2}(z'-w') \,\textup{d}z' \bigg)^2 \\
& \times (\mu\ast \mathbbm{g}_{t/\sqrt{2}})_{a_j\lambda,a_j\lambda^{-1}}(w-w') \Big(\prod_{\substack{1\leq k\leq n\\k\neq j}}(\mu\ast \mathbbm{g}_{t})_{a_k\lambda,a_k\lambda^{-1}}(u_k)\Big) 
\,\textup{d}w \,\textup{d}w' \,\textup{d}(u_k)_{k\neq j} \,\frac{\textup{d}t}{t} .
\end{align*}

We are controlling $\mathcal{A}$ rather crudely. 
First, we eliminate the convolution $\mu\ast \mathbbm{g}_{t/\sqrt{2}}$ by integrating out $w'$.
Then we bound
\begin{align*}
& \int_{\R^2} \bigg( \int_{\R^2} \mathbbm{1}_A(z) \Big(\prod_{\substack{1\leq k\leq n\\k\neq j}}\mathbbm{1}_A(z+u_k)\Big) 
\mathbbm{h}^{(l)}_{a_jt\lambda/2,a_jt\lambda^{-1}/2}(z-w) \,\textup{d}z \bigg)^2 \,\textup{d}w \\
& \leq \Big\| \mathbbm{1}_A \ast \big|\mathbbm{h}^{(l)}_{a_jt\lambda/2,a_jt\lambda^{-1}/2}\big| \Big\|_{\textup{L}^2(\R^2)}^2 
\leq \|\mathbbm{1}_A\|_{\textup{L}^2(\R^2)}^2 \|\mathbbm{h}^{(l)}\|_{\textup{L}^1(\R^2)}^2 \lesssim |A|,
\end{align*}
where we used Young's inequality \eqref{eq:Youngineq}.
Finally, we integrate in the remaining variables to get
\[ \mathcal{A} \lesssim |A| \log\frac{1}{\varepsilon}. \]
When dealing with $\mathcal{B}$ we begin by integrating out the variables $w$ and $u_k$ for each $k\neq j$. 
Then we apply \eqref{eq:convidaa} and Plancherel's theorem to what is left, obtaining
\[ \mathcal{B} \leq -\frac{1}{2} \sum_{j=1}^{n} \int_{\varepsilon}^{1} \int_{\R^2} \big|\widehat{\mathbbm{1}_A}(\zeta,\eta)\big|^2 \,\widehat{\mathbbm{k}}\Big(\frac{a_j t\zeta}{\sqrt{2}} \lambda, \frac{a_j t\eta}{\sqrt{2}} \frac{1}{\lambda}\Big) \,\textup{d}(\zeta,\eta) \,\frac{\textup{d}t}{t}. \]
This, \eqref{eq:justCS}, and Lemma \ref{lm:kscalesnicely} altogether imply
\begin{align*}
& \int_{0}^{\infty} \big(\mathcal{N}_{\lambda}^{\varepsilon}(A;R)-\mathcal{N}_{\lambda}^{1}(A;R)\big)^2 \,\frac{\textup{d}\lambda}{\lambda} \\
& \lesssim R^{-4} |A| \Big(\log\frac{1}{\varepsilon}\Big)
\sum_{j=1}^{n} \int_{\varepsilon}^{1} \int_{\R^2} \big|\widehat{\mathbbm{1}_A}(\zeta,\eta)\big|^2 \bigg( \int_{0}^{\infty}\Big|\widehat{\mathbbm{k}}\Big(\frac{a_j t\zeta}{\sqrt{2}} \lambda, \frac{a_j t\eta}{\sqrt{2}} \frac{1}{\lambda}\Big)\Big| \,\frac{\textup{d}\lambda}{\lambda} \bigg) \,\textup{d}(\zeta,\eta) \,\frac{\textup{d}t}{t} \\
& \lesssim R^{-4} |A| \Big(\log\frac{1}{\varepsilon}\Big)
\int_{\varepsilon}^{1} |A| \,\frac{\textup{d}t}{t} 
= \delta^2 \Big(\log\frac{1}{\varepsilon}\Big)^2.
\end{align*}
It remains to change the integration variable, $\lambda=e^{\alpha}$.
\end{proof}

\begin{proof}[Proof of Theorem \ref{thm:spacetimedensity}]
For a given $\delta\in(0,1]$, Lemmas \ref{lm:spacetimestruct} and \ref{lm:spacetimeunif} allow us to find $\varepsilon\sim\delta^{2n}$ such that
\[ \big|\mathcal{N}_{\lambda}^{0}(A;R)-\mathcal{N}_{\lambda}^{\varepsilon}(A;R)\big| \leq \frac{1}{3} \mathcal{N}_{\lambda}^{1}(A;R) \]
holds for every $1\leq\lambda\leq R$.
Afterwards, by taking 
\[ J \sim \frac{1}{\delta^{2n}} \Big(1+\log\frac{1}{\delta}\Big)^2 \]
and restricting the integration in Lemma \ref{lm:spacetimeerr} to $[0,J]$, we can find a particular $\beta\in[0,J]$ such that
\[ \big|\mathcal{N}_{e^{\beta}}^{\varepsilon}(A;R)-\mathcal{N}_{e^{\beta}}^{1}(A;R)\big| 
\leq \Big(\frac{1}{J}\int_0^J \big(\mathcal{N}_{e^{\alpha}}^{\varepsilon}(A;R)-\mathcal{N}_{e^{\alpha}}^{1}(A;R)\big)^2 \,\textup{d}\alpha\Big)^{1/2}
\leq \frac{1}{3} \mathcal{N}_{e^{\beta}}^{1}(A;R). \]
Decomposition \eqref{eq:decomposition} then shows 
\[ \mathcal{N}_{e^{\beta}}^{0}(A;R) \geq \frac{1}{3}\mathcal{N}_{e^{\beta}}^{1}(A;R) \gtrsim \delta^{n+1} > 0 \]
and we are done.
In the last display we needed $R\geq e^J\geq e^{\beta}$ in order to be able to apply Lemma \ref{lm:spacetimestruct}. 
This is certainly satisfied whenever 
\[ R \geq \exp\big(C\delta^{-2n}(1+\log\delta^{-1})^2\big) \]
for a sufficiently large constant $C>0$.
On the other hand, the condition from the theorem statement can be rewritten as
\[ R \geq \exp\big(C\delta^{-2n-1}), \]
which is a stronger requirement than needed.
\end{proof}


\section{Remarks for harmonic analysts}
\label{sec:remarksanalysts}
One might wonder about the analytical reason why a density result for area $1$ rectangles is not possible in any dimension, while we were able to prove such a result for area 1 right-angled triangles.

Namely, take Schwartz functions $\varphi$ and $\psi$, respectively defined on $\R^{d_1}$ and $\R^{d_2}$, that have compact frequency supports and suppose that at least one of these supports does not contain the origin.
Also, take some $4$-tuple $(p_1,p_2,p_3,p_4)$ of Lebesgue exponents in $[1,\infty]$ satisfying the H\"{o}lder scaling, 
\[ \sum_{k=1}^{4} \frac{1}{p_k} = 1 . \]
In what follows, $0<r<R$ will be certain parameters that determine the range of the scales $t$ that we want to study.
The best constant $C_{r,R}$ in the quadrilinear estimate
\begin{align*}
\Big| \int_{r}^{R} \int_{\R^{d_1}} \int_{\R^{d_1}} \int_{\R^{d_2}} \int_{\R^{d_2}} 
& f_1(x,y) f_2(x,y') f_3(x',y) f_4(x',y') \\
& \times\varphi_{t}(x-x') \psi_{1/t}(y-y') 
\,\textup{d}y \,\textup{d}y' \,\textup{d}x \,\textup{d}x' \,\frac{\textup{d}t}{t} \Big| 
\leq C_{r,R} \prod_{k=1}^{4} \|f_k\|_{\textup{L}^{p_k}(\R^{d_1+d_2})} 
\end{align*}
grows like a multiple of $\log(R/r)$ as $R/r\to\infty$. This is not any better than in the trivial estimate obtained from H\"{o}lder's inequality applied for each fixed $t$. 
Rather than giving a direct argument for this observation, one could argue that any nontrivial cancellation, i.e., a bound of the form $o(\log(R/r))$ as $R/r\to\infty$, would translate to a positive result for rectangles of area $1$. We would simply adjust the proof of Theorem \ref{thm:boxdensity} to work for fixed area rectangles, making it similar to the proof of Theorem \ref{thm:simpldensity}, which we know to be impossible due to Theorems \ref{thm:colorR2} and \ref{thm:colorRn}.

A direct argument goes as follows. For simplicity we work in two dimensions, i.e., $d_1=d_2=1$. For some $M>0$ take $f_1\colon\R^2\to\mathbb{C}$ to be
\[ f_1(x,y) := e^{2\pi \ii x y} \mathbbm{g}\Big(\frac{x}{M}\Big) \mathbbm{g}\Big(\frac{y}{M}\Big) \]
and let $f_2:=\overline{f_1}$, $f_3:=\overline{f_1}$, $f_4:=f_1$.
The right hand side of the quadrilinear estimate is comparable to $C_{r,R}M^2$, while
{\allowdisplaybreaks
\begin{align*}
& \lim_{M\to\infty} \frac{1}{M^2} \int_{r}^{R} \int_{\R^4} f_1(x,y) f_2(x,y') f_3(x',y) f_4(x',y') 
\varphi_{t}(x-x') \psi_{1/t}(y-y')  \,\textup{d}(x,x',y,y') \,\frac{\textup{d}t}{t} \\
& \quad\text{[substitute $u=x-x'$, $v=y-y'$]} \\
& = \lim_{M\to\infty} \frac{1}{4} \int_{r}^{R} \int_{\R^2} e^{2\pi \ii u v}
\varphi_{t}(u) \psi_{1/t}(v) \mathbbm{g}\Big(\frac{u}{M}\Big) \mathbbm{g}\Big(\frac{v}{M}\Big) \,\textup{d}(u,v) \,\frac{\textup{d}t}{t} \\
& = \frac{1}{4} \Big(\log\frac{R}{r}\Big) \int_{\R} \widehat{\varphi}(-v) \psi(v) \,\textup{d}v,
\end{align*}}
so we only need to make sure that $\varphi$ and $\psi$ are such that the last integral is nonzero.
The fact that the above quadrilinear form does not satisfy any nontrivial estimates is in sharp contrast with results for the so-called \emph{twisted paraproduct}, where the scales are aligned as $\varphi_{t},\psi_{t}$, rather than $\varphi_{t},\psi_{1/t}$. It was shown to be bounded, at least when the exponents $p_k$ are greater than $2$, by Durcik \cite{Dur15,Dur17} and, in a dyadic model, previously also by the author \cite{Kov12}.

The form associated with area $1/2$ triangles is obtained by taking $f_4$ to be constantly $1$ and the associated exponent $p_4$ to be $\infty$. The resulting trilinear inequality
\begin{align*} 
\Big| \int_{r}^{R} \int_{\R^{d_1}} \int_{\R^{d_1}} \int_{\R^{d_2}} \int_{\R^{d_2}} 
f_1(x,y) f_2(x,y') f_3(x',y)
\varphi_{t}(x-x') \psi_{1/t}(y-y') 
\,\textup{d}y \,\textup{d}y' \,\textup{d}x \,\textup{d}x' \,\frac{\textup{d}t}{t} \Big| & \\
\leq C'_{r,R} \prod_{k=1}^{3} \|f_k\|_{\textup{L}^{p_k}(\R^{d_1+d_2})} &
\end{align*}
now actually has non-trivial cancellation over the scales $t$. Its best constants $C'_{r,R}$ grow at most like $O((\log(R/r))^{1/2})$, which is easily seen similarly as in the proof of Lemma \ref{lm:simplerr}: by separating one of the functions and applying the classical square function bound to it.
However, one could still ask if the above trilinear form is also bounded as $r\to0+$ and $R\to\infty$, or, equivalently, if the constants $C'_{r,R}$ are even independent of the scale truncation parameters $r$ and $R$.
If, on the other hand, the scales were aligned as $\varphi_{t},\psi_{t}$, this would again be the bound for the twisted paraproduct and it would follow from \cite[Theorem 1]{Kov12}.

Finally, if only functions $f_1$ and $f_4$ remain, then it is indeed easy to use Plancherel's theorem and Lemma \ref{lm:kscalesnicely} to prove the estimate
\begin{align*}
\Big| \int_{r}^{R} \int_{\R^{d_1}} \int_{\R^{d_1}} \int_{\R^{d_2}} \int_{\R^{d_2}} 
f_1(x,y) f_4(x',y') \varphi_{t}(x-x') \psi_{1/t}(y-y') 
\,\textup{d}y \,\textup{d}y' \,\textup{d}x \,\textup{d}x' \,\frac{\textup{d}t}{t} \Big| & \\
= \Big| \int_{r}^{R} \big\langle f_1 , f_4 \ast (\varphi_t\otimes\psi_{1/t}) \big\rangle_{\textup{L}^2(\R^{d_1+d_2})} \,\frac{\textup{d}t}{t} \Big|
\leq C \|f_1\|_{\textup{L}^{2}(\R^{d_1+d_2})} \|f_4\|_{\textup{L}^{2}(\R^{d_1+d_2})} & ,
\end{align*}
which is now uniform in $r$ and $R$.
A manifestation of this property has been used in the proof of Theorem \ref{thm:spacetimedensity}: in the proof of Lemma \ref{lm:spacetimeerr} we were able to integrate over all scales $e^{\alpha}$, even without squaring the difference first.

We intentionally do not discuss this matter further or in more detail here. It could be interesting on its own to define a class of ``volume-preserving'' or ``time reversed'' (alluding to the appearance of scale $1/t$) multilinear singular integral operators, and then attempt to classify those that are bounded and those that merely possess non-trivial cancellation over the scales.

\section{Tabular summary}
\label{sec:summary}
This section offers a condensed landscape of most (but not all) density results discussed in the present paper. Our goal is to emphasize distinctions among different point configurations and different types of ``largeness'' conditions imposed on the containing set.

\begin{table}
\centering
\begin{tabular}{|c|c||c|c|c|c|}
\hline
\multicolumn{2}{|c||}{\multirow{2}{*}{\textbf{configuration}\textbackslash\textbf{set}}} & \multicolumn{2}{c|}{positive upper density} & \multicolumn{2}{c|}{infinite measure} \\ 
\cline{3-6} 
\multicolumn{2}{|c||}{} & minimal dim. & increased dim. & minimal dim. & increased dim. \\ \hline \hline
\multirow{2}{*}{segment} 
& all & \multicolumn{4}{|c|}{ A (\cite{Soi24:book}) } \\ \cline{2-6}
& large & A\footnotemark & U (\cite{FKW90:dist,FM86:dist,B86:roth}) & \multicolumn{2}{|c|}{A\footnotemark} \\ \hline
\multicolumn{2}{|c||}{right triangle} & \multicolumn{4}{|c|}{ U (\cite{Koi25}) } \\ \hline
\multicolumn{2}{|c||}{simplex} & \multicolumn{4}{|c|}{ U (\cite{Erd78}) } \\ \hline
\multirow{2}{*}{box} 
& all & \multicolumn{4}{|c|}{ A (Thm.\,\ref{thm:colorRn}) } \\ \cline{2-6}
& large & ? ($\approx$Prob.\,\ref{prob:5}) & U (\cite{LM16:prod},\,Thm.\,\ref{thm:boxdensity}) & A (Thm.\,\ref{thm:parallel}) & ? \\ \hline
\multicolumn{2}{|c||}{parallelotope} & \multicolumn{2}{|c|}{? ($\approx$Prob.\,\ref{prob:parallel})} & A (Thm.\,\ref{thm:parallel}) & ? \\ \hline
\multirow{2}{*}{hypercube graph} 
& all & A (Thm.\,\ref{thm:hypercubes}) & ? & A (Thm.\,\ref{thm:hypercubes}) & ?\footnotemark \\ \cline{2-6}
& large & U (\cite{KP23}) & U (\cite{LM20}) & \multicolumn{2}{|c|}{?} \\ \hline
\multicolumn{2}{|c||}{hyperbolic segment} & \multicolumn{2}{|c|}{U (\cite{BMK17},\,Thm.\,\ref{thm:bmk})} & \multicolumn{2}{|c|}{A (Rem.\,\ref{rem:onstars})} \\ \hline
\multicolumn{2}{|c||}{hyperbolic star} & \multicolumn{2}{|c|}{U (Thm.\,\ref{thm:spacetimedensity})} & \multicolumn{2}{|c|}{A (Rem.\,\ref{rem:onstars})} \\ \hline
\end{tabular}
\caption{Landscape of results.}
\label{tab:landscape}
\end{table}

In Table \ref{tab:landscape} we discuss configurations that are vertex-sets of a segment, a right triangle, a simplex, a rectangular box, or a parallelotope, each of a fixed length/area/volume. Additionally, we consider a hypercube graph with a prescribed product of edge-lengths, a hyperbolic unit segment, i.e., a pair of points the difference of which lies on the standard hyperbola, or a hyperbolic star made of several hyperbolic segments.
We mark whether they can be found in every set of positive upper density, or already in every set of infinite measure. 
When we say that a configuration is \emph{unavoidable} (U), this means that it is found in all sets that are large in the appropriate sense.
Otherwise a configuration is \emph{avoidable} (A). 
Sometimes we make a further distinction by trying to find configurations that span polytopes of either \emph{all} prescribed volumes or merely only sufficiently \emph{large} ones.
Moreover, different results are sometimes available depending on whether the ambient dimension is the \emph{minimal} one that can host the configuration, or an \emph{increased} one. 
The reader should understand that an isolated question mark (?) simply means that the author is not aware of any relevant literature, including the possibility that further investigation reveals the entry as being easy.
The open problems listed in the introduction have already attracted interest, having been raised on several previous occasions.

\footnotetext[1]{A trivial one-dimensional set avoiding all half-integer distances is the union of intervals $[k-1/5,k+1/5]$ centered at integers $k$.}
\footnotetext[2]{In any number of dimensions we can simply take a union of balls of small radius $\varepsilon>0$ around the points $(k,0,\ldots,0)$, $k\in\mathbb{Z}$.}
\footnotetext[3]{For the particular case of the $1$-skeleton of the $2$-cube with prescribed product of edge-lengths $a_1 a_2$, an avoiding set of infinite measure in $\mathbb{R}^n$ can be constructed as a union of balls around integer points $(k_1,\ldots,k_n)\in\mathbb{Z}^n$ with radii $\varepsilon/(|k_1|+\cdots+|k_n|+1)$, where $\varepsilon>0$ is sufficiently small.}


\section*{Acknowledgments}
This work was supported in part by the Croatian Science Foundation under the project HRZZ-IP-2022-10-5116 (FANAP). The author is grateful to Mohammad Bardestani, Adrian Beker, James Davies, Polona Durcik, Keivan Mallahi-Karai, and Rudi Mrazovi\'{c} for several very useful discussions, and to Boris Alexeev for helping formalize the proof of Theorem \ref{thm:colorR2} in Lean. Finally, the author thanks the anonymous referees for numerous suggestions on improving the presentation.


\bibliography{monorectangles}{}
\bibliographystyle{plainurl}

\end{document}